\documentclass{amsart}
\usepackage{amscd,amssymb,amsopn,amsmath,amsthm,graphics,amsfonts,accents,enumerate,verbatim,calc}
\usepackage[dvips]{graphicx}
\usepackage[colorlinks=true,linkcolor=red,citecolor=blue]{hyperref}
\usepackage[all]{xy}
\usepackage{cite}
\usepackage{tikz}
\usepackage{tikz-cd}
\usepackage{mathrsfs}

\calclayout

\newcommand{\rt}{\rightarrow}

\newcommand{\st}{\stackrel}

\newcommand{\al}{\alpha}

\newcommand{\La}{\Lambda}
\newcommand{\Ga}{\Gamma}

\newcommand{\Z}{\mathbb{Z}}

\newcommand{\SA}{\mathscr{A}}

\newcommand{\SC}{\mathscr{C}}
\newcommand{\SE}{\mathscr{E}}

\newcommand{\SG}{\mathscr{G}}

\newcommand{\SX}{\mathscr{X}}

\newcommand{\CC}{\mathcal{C} }

\newcommand{\CG}{\mathcal{G} }
\newcommand{\CH}{\mathcal{H}}

\newcommand{\CQ}{\mathcal{Q} }

\newcommand{\CS}{\mathcal{S} }
\newcommand{\CT}{\mathcal{T} }

\newcommand{\CX}{\mathcal{X} }

\newcommand{\mmod}{{\rm{{mod\mbox{-}}}}}

\newcommand{\prj}{{\rm{prj}\mbox{-}}}

\newcommand{\Gprj}{{\Gp\mbox{-}}}

\newcommand{\ind}{{\rm{ind}}}
\newcommand{\Tr}{{\rm{Tr}}}

\newcommand{\Gp}{{\rm{Gprj}}}

\newcommand{\Ker}{{\rm{Ker}}}

\newcommand{\Hom}{{\rm{Hom}}}
\newcommand{\Ext}{{\rm{Ext}}}

\usepackage[symbol]{footmisc}

\theoremstyle{plain}
\newtheorem{theorem}{Theorem}[section]
\newtheorem{corollary}[theorem]{Corollary}
\newtheorem{lemma}[theorem]{Lemma}

\newtheorem{proposition}[theorem]{Proposition}

\newtheorem{notation}[theorem]{Notation}

\theoremstyle{definition}
\newtheorem{definition}[theorem]{Definition}
\newtheorem{example}[theorem]{Example}

\newtheorem{remark}[theorem]{Remark}
\newtheorem{setup}[theorem]{Setup}
\newtheorem{s}[theorem]{}

\theoremstyle{plain}

\theoremstyle{definition}

\numberwithin{equation}{section}

\begin{document}
\title[Stable Auslander-Reiten components]{On the stable Auslander-Reiten components of \\ certain monomorphism categories}

\author[Rasool Hafezi]{Rasool Hafezi${}^{\ast}$}\footnotetext{${}^{\ast}$Corresponding author.}
\address{School of Mathematics and Statistics, Nanjing University of Information Science \& Technology, Nanjing, Jiangsu 210044, P.\,R. China}
\email{hafezi@nuist.edu.cn}

\author[Yi Zhang]{Yi Zhang}
\address{School of Mathematics and Statistics, Nanjing University of Information Science \& Technology, Nanjing, Jiangsu 210044, P.\,R. China}
\email{zhangy2016@nuist.edu.cn}

\makeatletter\@namedef{subjclassname@2020}{\textup{2020} Mathematics Subject Classification} \makeatother
\subjclass[2020]{16G20,16E65,16G70, 16G10.}

\keywords{Monomorphism category, almost split sequence, Auslander-Reiten quiver, Gorenstein projective module.}

\begin{abstract}
Let $\Lambda$ be an Artin algebra and let $\rm{Gprj}\mbox{-}\Lambda$ denote the class of all finitely generated Gorenstein projective $\Lambda$-modules. In this paper, we study the components of the stable Auslander-Reiten quiver of a certain subcategory of the monomorphism category $\mathcal{S}({\rm Gprj}\mbox{-}\Lambda)$ containing boundary vertices. We describe the shape of such components. It is shown that certain  components are linked to the orbits of an auto-equivalence on the stable category $\underline{\rm{Gprj}}\mbox{-}\Lambda$.  In particular, for the finite components, we show that under certain mild conditions their cardinalities are divisible by  $3$. We see that this three-periodicity phenomenon reoccurs several times in the paper.
\end{abstract}

\maketitle

\tableofcontents

\allowdisplaybreaks

\section{Introduction}
\s Let $\La$ be an Artin algebra. Almost split sequences, also called Auslander-Reiten sequences, are the building blocks of the Auslander-Reiten theory, AR-theory for short. These non-split short exact sequences contain information about how certain morphisms factor. The theory developed in \cite{AR2, AR3}  provides an efficient tool for studying the local structure of additive categories. The existence of almost split sequences for $\mmod\La$, the category of finitely generated $\La$-modules, was first proved in \cite{AR}.

Because of the importance of the theory, it has been studied in many different settings, such as  subcategories of abelian categories \cite{AS},  triangulated categories \cite{Ha},  exact categories \cite{LNP}, extriangulated categories \cite{INP},  Krull-Schmidt categories \cite{Li}, and  submodule categories \cite{RS}.

The theory became more interesting and became the core of the study in modern representation theory after Ringel embedded the almost split sequences into certain quivers known as Auslander-Reiten quivers \cite{Ri}. The AR-quiver of a category $\SC$, usually denoted by $\Ga(\SC)$ or $\Ga_{\SC}$, is a valued translation quiver \cite[\S IV.4]{ASS} that contains much of the  homological and combinatorial information of the category. Its vertices are the isoclasses $[X]$ of indecomposable objects in $\SC$ and arrows are defined by  using a notion of irreducibility.

The shape of the AR-quiver of an algebra not only visualizes the information of the category $\SC$ but also gives valued information on the category. For instance, one  important class of Artin algebras is Artin algebras of finite representation type. Recall that an algebra $\La$ is  of finite representation type if, up to isomorphism, there are only finitely many indecomposable modules in $\mmod\La$. Their module categories are well described by the AR-quiver, see \cite{Riee}. For instance,  Riedtmann classified all selfinjective algebras of finite representation type by their Auslander-Reiten quivers \cite{Rie, Rie1}. If $\SC$ is a Frobenius exact category, then the stable Auslander-Reiten quiver of $\SC$, denoted by $\Ga^s_{\SC}$, is obtained by removing all projective-injective vertices of $\Ga_{\SC}$.

We point out that the AR-quiver often decomposes into a union of (infinite) components and the structures of such components are studied extensively. For example, in \cite{BE}, tree classes of the components of the stable Auslander-Reiten quiver of a quantum complete intersection were completely described. The shape of the components of the stable Auslander-Reiten quiver that contain Heller lattices was determined in \cite{AKM} and  components of the stable Auslander-Reiten quiver ${\Ga }^{s}_{\SC}$ of a finite group scheme $\CG$ over a field $k$ of characteristic $p$ was studied in \cite{T}.

\s The study of monomorphism categories, also known  as submodule categories, almost has the  same age as AR-theory, or rather older. The study of such categories goes back to Birkhoff \cite{B}. Monomorphism categories allow us to apply methods from homological algebra, combinatorics and geometry to open problems in linear algebra. The monomorphism category of $\La$, denoted by $\CS(\La)$, consists of all monomorphisms in $\mmod\La$ as objects. The morphisms are defined by commutative diagrams. The study of such categories has been the subject of several recent researches, see e.g. \cite{ RS3, RS, RS2, RZ, S, Z}. For the history and details of the theory,  please refer to  \cite{RZ, ZX} and references therein.

In \cite{CH} it is proved  that the monomorphism category of a Frobenius abelian category is a Frobenius exact category. In \cite{RS} it is shown that $\CS(\La)$ has Auslander-Reiten sequences. The authors of \cite{XZZ} generalized results of \cite{RS} from monomorphism categories to the category $\CS_n(\La)$, that is,  the category of all sequences $A_1\rt A_2\rt\cdots \rt A_n$ in $\mmod\La$, such that all morphisms are monomorphisms. Monomorphism categories associated to arbitrary species also had been studied in \cite{GKSP}. In \cite{H}, the first  author of this paper investigated the monomorphism categories associated to a subcategory.

\s Auslander and Bridger \cite{AB} introduced finitely generated modules of Gorenstein dimension zero as a natural generalization of finitely generated projective modules.
This notion was further developed  by several authors and a definition in full generality for arbitrary modules over associative and unitary rings was proposed by Enochs and Jenda \cite{EJ}, by using the terminology of  Gorenstein projective modules.

It has been proved that Gorenstein projective modules are strongly related with monomorphism categories. For instance, Gorenstein projective bimodules over the tensor product of two algebras are studied via monomorphism categories \cite{HLXZ}. Quite recently in \cite{WL}, it is shown that there are recollements induced by the monomorphism categories of Gorenstein projective and Ding projective modules.

\s Let $\CS_{\SG}(\La)$ denote the subcategory of $\CS(\La)$ consisting of all morphisms $f:X\rt Y$ such that all $X, Y$ and ${\rm Cok}(f)$ lie in $\Gprj\La$. This monomorphism subcategory  will be called the {$\CG$}-monomorphism category of $\La$. In this paper we are mainly interested in studying the shape of the components of the stable Auslander-Reiten quiver of $\CS_{\SG}(\La)$. 

The paper is structured as follows. Section \ref{Preliminary} is devoted to prepare the necessary background for our main results. In particular, Frobenius categories,     Auslander-Reiten sequences, Auslander-Reiten quivers and the category of Gorenstein projective modules are studied. In Section \ref{classificationGorProj}, we introduce and study the $\CG$-monomorphism category $\CS_{\SG}(\La)$ of $\La$. In particular, we see that if  $\Gprj\La$ is a contravariantly finite subcategory of $\mmod \La$, then $\CS_{\SG}(\La)$ has almost split sequences. Then we study the structure of almost split sequences starting with/ending at certain objects in $\CS_{\SG}(\La)$. We also investigate the structure of the middle terms of certain almost split sequences in $\CS_{\SG}(\La)$.

In Section \ref{stable componnets} we define an auto-equivalence $\vartheta$ on $\underline{\rm Gprj}\mbox{-}\La$, the stable category of Gorenstein projective modules, and study the $\vartheta$-orbits of indecomposable non-projective Gorenstein projective modules. As a generalization of the boundary modules  \cite[(5.1)]{RS2}, we define the notion of boundary vertices of $\Ga^s_{\CS}$, the stable Auslander-Reiten quiver of $\CS_{\SG}(\La)$, and study those components of $\Ga^s_{\CS}$ that contain a boundary vertex.
In particular, when $\La$ is a CM-finite algebra and $\Ga$ is a component of $\Ga^s_{\CS}$ containing a boundary vertex, the shape of $\Ga$ both in the finite and in the infinite case will be discussed, see Theorem \ref{Proposition 2.7}.
If we consider the boundary vertices of the form $(0 \rt G)$, where $G$ runs over the set of all pairwise non-isomorphic indecomposable non-projective Gorenstein projective modules, then there is a well-defined map between the set of certain $\vartheta$-orbits and the set of all components of the stable Auslander-Reiten quiver $\Ga^s_{\CS}$ containing vertices of the form $(0\rt  G)$, see Proposition \ref{proposition 3.7}.

In Section \ref{Sec 5} we study components of  $\Ga^s_{\CS}$ with finite cardinality containing a boundary vertex. In particular, we show that if $G$ is an indecomposable non-projective Gorenstein projective module such that the component of $\Ga^s_{\CS}$ containing a boundary vertex $(0 \rt G)$ is finite, then this component contains at most two distinct $\tau_{\CS}$-orbits which are different from the $\tau_{\CS}$-orbit of the boundary vertex $(0 \rt G)$, see Theorem \ref{Proposition 2.13}.

In the last section of the paper, we investigate an interesting three-periodicity phenomenon. We show that under some mild conditions, the cardinality of the finite components of the stable Auslander-Reiten quiver $\Ga^s_{\CS}$ containing a boundary vertex is divisible  by 3, see Theorem \ref{Proposition 4.8}.
As it was mentioned by the referee, this kind of three-periodicity phenomenon, for some of the objects, is happening throughout the paper, see Remark \ref{referee}.

\section{Preliminaries}\label{Preliminary}
Throughout this paper, $\La$ will always denote an Artin algebra, and $\mmod \La$ will denote the category of finitely generated right $\La$-modules. By a subcategory we always mean a full subcategory. The subcategory of $\mmod\La$ consisting of all projective modules is denoted by $\prj\La$. An additive subcategory $\SX$ of $\mmod\La$ is called a resolving subcategory if it contains $\prj\La$ and if it is closed under extensions and kernels of epimorphisms.

\s {\sc Functorially finite subcategories.} Let $\SX$ be an additive subcategory of an abelian category $\SA$. Let $A \in \SA$ be an object. A right $\SX$-approximation of $A$ is a morphism $\varphi: X \rt A$, with $X \in \SX$ satisfying the property that any other morphism $\varphi': X' \rt A$ with $X' \in \SX$, factors through $\varphi$. Dually one can define the notion of a left $\SX$-approximation of $A$. We say that $\SX$ is a contravariantly finite subcategory, resp. a covariantly finite subcategory, of $\SA$ if any object $A$ in $\SA$ admits a right, resp. a left, $\SX$-approximation. We say that $\SX$ is a functorially finite subcategory of $\SA$ if it is both a contravariantly finite and a covariantly finite subcategory of $\SA$.

\begin{remark}\label{Rmk: KS}
Let $\SX \subseteq \mmod\La$ be a contravariantly finite subcategory. It is proved in \cite[Corollary 2.6]{KS} that if $\SX$ is  a resolving subcategory additionally, then it is also covariantly finite and hence is a functorially finite subcategory of $\mmod\La$.
\end{remark}

\s {\sc Frobenius categories.}
Let $\SC$ be an additive category. A pair $(i,p)$ of morphisms in $\SC$, depicted by $X \st{i}{\rt} Y \st{p}{\rt} Z$, is called a kernel-cokernel pair if $i$ is a kernel of $p$ and $p$ is a cokernel of $i$. Let $\SE$ be a fixed class of kernel-cokernel pairs in $\SC$ which is closed under isomorphism and satisfies certain axioms \cite[Definition 2.1]{Bu}. Then $(\SC, \SE)$ is called an exact category \cite{Q} and the distinguished kernel-cokernel pairs in $\SE$ are called conflations.  In \cite[Appendix A]{K}, it is shown that exact categories are just full and extension-closed subcategories of abelian categories. For the other  terminologies, we will follow \cite{GR, K}.

Let $(\SC, \SE)$ be an exact category. An object $P \in \SC$ is called a projective object relative to $\SE$, or simply an $\SE$-projective object, if the functor $\Hom_{\SC}(P, - )$ from $\SC$ to $\SA b$, the category of abelian groups, is exact. We say that $\SC$ has enough projective objects if for every object $X \in \SC$ there exists an admissible epimorphism $P \rt X$ such that $P$ is $\SE$-projective. The notions of $\SE$-injective objects and a category having enough $\SE$-injective objects are defined dually.

An exact category $(\SC, \SE)$ is called Frobenius if it has enough $\SE$-projective and enough $\SE$-injective objects and the classes of $\SE$-projective and $\SE$-injective objects coincide \cite[\S 13.4]{Bu}. Happel \cite[Theorem 2.6]{Ha} observed that the stable category $\underline{\SC}$ of a Frobenius category is a triangulated category. Recall that $\underline{\SC}$ is the category consisting of the same objects as $\SC$. For two objects $\underline{X}, \underline{Y} \in \underline{\SC}$ the morphism set $\Hom_{\underline{\SC}}(\underline{X}, \underline{Y})$ is defined to be $\Hom_{\SC}(X,Y)$ modulo the subgroup of $\Hom_{\SC}(X,Y)$ consisting of all morphisms that factors through an $\SE$-projective object. The stable Hom set is also denoted by $\underline{\Hom}$.

\s {\sc Auslander-Reiten sequences.}
Let $\SC$ be an additive category. A morphism $p: Y \rt Z$ in $\SC$ is said to be right minimal if every morphism $h: Y \rt Y$ with $ph=p$ is an isomorphism.  A morphism $p: Y \rt Z$ in $\SC$ is called a right almost split morphism if it is not a retraction and each morphism $g : W \rt Z$ which is not a retraction factors through $p$. A right minimal almost split morphism is a morphism which is right minimal and right almost split in $\SC$. Dually, left minimal, left almost split and left minimal almost split morphisms in $\SC$ are defined.

Let $(\SC, \SE)$ be an exact category. A conflation $\eta: 0 \rt X \st{i}{\rt} Y \st{p}{\rt} Z \rt 0$ in $\SE$ is called an almost split conflation, or almost split sequence, if $i$ is a left minimal almost split morphism and $p$ is a right minimal almost split morphism. In this case, $\eta$ is called an almost split sequence starting with $X$ or an almost split sequence ending with $Z$. Throughout, we call $X$, $Y$ and $Z$, the left, middle and right term of $\eta$, respectively.

We say that $\SC$ has almost split sequences if each non-projective indecomposable object appears as the right term of an almost split sequence, and  each non-injective indecomposable object appears as the left term of an almost split sequence.  In addition, each projective indecomposable object should appear as the right term of a right minimal almost split morphism and each injective indecomposable object should appear as the left term of a left minimal almost split morphism.

\begin{remark}
Let $\SC=\mmod\La$ and let $\tau_{\La} = D\Tr$ be the Auslander-Reiten translation, where $D$ is the standard duality functor and $\Tr$ stands for the transpose. It is known that the end terms of an almost split sequence are related to $\tau_{\La}$ and $\tau_{\La}^{-1}=\Tr D$, see e.g.  \cite[\S IV.2]{ASS}.

In view of this property, over an exact category $(\SC, \SE)$ that has almost split sequences, we define $\tau_{\SC}$ by saying that $\tau_{\SC} Z = X$ if and only if there exists an almost split sequence $0 \rt X \st{i}\rt Y \st{p}\rt Z \rt 0$ in $\SE$.
\end{remark}

\begin{remark}\label{AS}
One of the old problems in representation theory of algebras is to determine which categories admit almost split sequences. One of the important results in this direction is due to Auslander and Smal{\o}  \cite[Theorem 2.4]{AS}, showing that if $\SC$ is an extension closed and functorially finite subcategory of $\mmod\La$, then $\SC$ has almost split sequences. Note that $\SC$, as an extension closed full subcategory of the abelian category $\mmod\La$, is an exact category.
\end{remark}

\s {\sc Auslander-Reiten quivers.}
Recall that a quiver $\Delta$ is a quadruple $(\Delta_0, \Delta_1, s, t)$, where $\Delta_0$, resp. $\Delta_1$, is the set of vertices, resp. the set of arrows, of $\Delta$, and $s, t: \Delta_1 \rt \Delta_0$ are two maps, called the source and target maps, which associate each arrow $\alpha$ to its source $s(\alpha)$ and its target $t(\alpha)$, respectively. Sometimes, when maps $s$ and $t$ are clear from the context, a quiver will be denoted just by the sets of vertices and arrows.

Let $\SC$ be an exact category that has almost split sequences. The Auslander-Reiten quiver of $\SC$ is a quiver recording some important information of $\SC$ in the form of a quiver. It is a valued translation quiver \cite[\S IV.4]{ASS} whose vertices are modules in $\ind\SC$, the subcategory of isomorphism classes of indecomposable objects in $\SC$. For two vertices $[X]$ and $[Y]$ in $\ind\SC$, the arrows $[X] \rt [Y]$ are defined by using the notion of irreducible morphisms. These valued arrows of $\Ga_{\SC}$ describe the minimal left almost split morphisms with indecomposable domains and the  minimal right almost split morphisms with indecomposable codomains. The translation is given by the Auslander-Reiten translations $\tau_{\SC}$ and $\tau^{-1}_{\SC}$. The connected components of $\Ga_{\SC}$ are called Auslander-Reiten components of $\CC$. Throughout this paper, the Auslander-Reiten quiver of $\SC$ will be denoted by $\Ga_{\SC}$.

If $[X]$ is a vertex corresponding to an indecomposable projective object $X$, resp. indecomposable injective object $X$, then $[X]$ is called a projective vertex, respectively an injective vertex, of $\Ga_{\SC}$. If $\SC$ is a Frobenius exact category, then the stable Auslander-Reiten quiver of $\SC$ is denoted by $\Ga^s_{\SC}$. It is obtained by removing all projective-injective vertices  of  $\Ga_{\SC}$ and hence removing all arrows attached to those vertices.

If $\SC=\mmod\La$, then $\Ga_{\mmod\La}$, resp. $\Ga^s_{\mmod\La}$, will be denoted by $\Ga_{\La}$, resp. $\Ga^s_{\La}$.

We omit the valuations of the arrows in the Auslander-Reiten quivers if they are not essential. We often identify the isomorphism class of an indecomposable object with its representative, or with the corresponding vertex in the  Auslander-Reiten quiver.

For background on the AR-theory of Artin algebras the reader is referred to one of the books \cite{ARS, ASS} or \cite{SY}. For details of the AR-theory of Krull-Schmidt categories we refer the reader to \cite{Li}.

\s {\sc Gorenstein projective modules.}
Let $\La$ be an Artin algebra. A $\La$-module $G$ is called Gorenstein projective if there exists an exact sequence
\[P^\bullet:\cdots\rightarrow P^{-1}\xrightarrow{d^{-1}} P^0\xrightarrow{d^0}P^1\rightarrow \cdots\]
of finitely generated projective modules which is $\Hom_{\La}( - , \La)$-exact, and $G \cong \Ker d^0$.  The subcategory of $\mmod\La$ consisting of all Gorenstein projective modules will be denoted by $\Gprj\La$.  In the following remark, we list some of the basic properties of the category $\Gprj\La$. The proof of these facts could be found  in \cite{Be, CH2}. Let $\Gprj\La^{\perp}$ be the subcategory
\[\{M \in \mmod\La \mid \Ext^1_{\La}(G, M) = 0 \  {\rm for \ any } \ G \in \Gprj\La \}\]
 of $\mmod \La$.

\begin{remark}\label{Prop. Gprj}
Let $\La$ be an Artin algebra and let
$\Gprj\La$ be the category of finitely generated Gorenstein projective $\La$-modules.
\begin{enumerate}
  \item $\Ext^1_{\La}(G, P)=0$, for every Gorenstein projective module $G$ and every projective module $P$.
  \item $\Gprj\La \cap \Gprj\La^{\perp} = \prj\La$.
  \item $\Gprj\La$ is a resolving subcategory of $\mmod\La$.
  \item $\Gprj\La$, as the full subcategory of $\mmod\La$, is closed under extensions. So it inherits a canonical exact structure from $\mmod\La$. Its conflations are all short exact sequences in $\mmod\La$ with all terms in $\Gprj\La$.
  \item  $\Gprj\La$ is a Frobenius exact category: its class of (relative) projective objects is exactly the class of projective modules in $\mmod\La$. This class also plays the role of relative injectives of the exact category $\Gprj\La$, see e.g. \cite[Proposition 2.1.11]{CH2}.
  \item By Remark \ref{AS}, $\Gprj\La$ has almost split sequences if it is a functorially finite subcategory of $\mmod\La$. This is the case, for instance, when $\La$ is a Gorenstein algebra, i.e. the injective dimension of $\La$ both as a left and as a right $\La$-module is finite \cite[Corollary 2.3.6]{CH2}. In this case, the Auslander-Reiten translation of $M$ in $\Gprj\La$ will be denoted by $\tau_{\SG}$.
  \item  An Artin algebra $\La$ is called selfinjective if $\La$ is injective as a $\La$-module, or equivalently, the projective and injective modules in $\mmod\La$ coincide. Obviously, when $\La$ is a selfinjective algebra, all modules are Gorenstein projective, that is, $\Gprj\La = \mmod \La$.
  \item  An Artin algebra $\La$ is said to be of finite Cohen-Macaulay type, CM-finite for short, if $\Gprj\La$ is of finite type, that is, up to isomorphism, there are only finitely many indecomposable modules in $\Gprj\La$, see \cite{Be}. It is clear that if $\La$ is CM-finite, then $\Gprj\La \subseteq \mmod\La$ is functorially finite.
\end{enumerate}
\end{remark}

Let $M \in \mmod\La$. Consider the short exact sequence
\begin{equation}\label{Syzygy}
0 \rt \Omega_{\La}(M)\st{i_M} \rt P_M \st{p_M} \rt M \rt 0
\end{equation}
in $\mmod \La$, where $P_M$ is the projective cover of $M$. The module $\Omega_{\La}M$ is called the first syzygy of $M$. Inductively we can define $\Omega^n_{\La}M$, the $n$-th syzygy of $M$, where $n \geq 2$ is an integer.

If we assume further that $M$ is a Gorenstein projective module, then we can define inductively $\Omega^n_{\La}(M)$, for negative integers, using the notion of minimal left $\prj\La$-approximation. We need the following property of the syzygy modules.

\begin{lemma}
Let $M$ be an indecomposable non-projective Gorenstein projective module. Then
\[\Omega_{\La}^{-1}\Omega_{\La}M \simeq M \simeq \Omega_{\La}\Omega^{-1}_{\La}M.\]
\end{lemma}

\begin{proof}
By Remark \ref{Prop. Gprj}, $\Ext^1_{\La}(G, P)=0$ for every Gorenstein projective module $G$ and every projective module $P$. So we can deduce that the morphism $i_m$ in the short exact sequence \eqref{Syzygy} is a left $\prj\La$-approximation. It is minimal because $M$ is indecomposable. So we obtain the first isomorphism. The second isomorphism is proved dually.
\end{proof}

Note that when $\La$ is a selfinjective algebra, $\Omega^n_{\La}(M)$, with $n$ a negative integer, is the same as the $n$-th cosyzygy of $M$ in the usual sense, which is defined by using injective envelopes.

\section{Relative monomorphism categories}\label{classificationGorProj}
Let $\La$ be an Artin algebra. The morphism category  $\CH(\La)$ of $\mmod\La$, is a category whose objects are morphisms $f: X \rt Y$ in $\mmod\La$, denoted by $(X\st{f}\rt Y)$ as an object in ${\CH}(\La)$. A morphism $\al: (X\st{f}\rt Y) \rt (X'\st{f'}\rt Y')$ in $\CH(\La)$, is a pair of morphisms $(\al_1 \ \al_2)$, where $\al_1: X \rt X'$ and $\al_2: Y \rt Y'$ are morphisms in $\mmod\La$ and $\al_2f=f'\al_1$. It is easy to see that $\CH(\La)$ is an abelian category. Moreover, by \cite[Theorem III.1.5]{ARS}, it is equivalent to the module category $\mmod T_{2}(\La)$, where $T_{2}(\La)$ is the lower triangular $2 \times 2$ matrix algebra over $\La$.

\s Let $\SX$ be a resolving subcategory of $\mmod\La$. The relative monomorphism category of $\La$ with respect to $\CX$, denoted by $\CS_{\SX}(\La)$, is defined to be the subcategory of $\CH(\La)$ consisting of all monomorphisms $f: X \rt Y$ in $\CH(\La)$ such that all $X, Y$ and the cokernel of $f$, ${\rm Cok}(f)$, lie in $\SX$. If we assume that $\SX=\mmod\La$, then the category $\CS_{\mmod\La}(\La)$ is just $\CS(\La)$, the submodule category of $\CH(\La)$ that consists of all monomorphisms in $\mmod\La.$ It  is also known as the monomorphism category of $\La$.

Of particular interest for us is when $\SX=\Gprj\La$.

\begin{definition}
Let $\La$ be an Artin algebra. The relative monomorphism category of $\La$ with respect to $\Gprj\La$ is the subcategory of $\CH(\La)$ consisting of all monomorphisms $f: X \rt Y$  such that all $X, Y$ and ${\rm Cok}(f)$ lie in $\Gprj\La$. According to our notation $\CS_{\Gprj\La}(\La)$ denotes the relative monomorphism category of $\La$ with respect to $\Gprj\La$. For ease of notation, throughout the paper instead of  $\CS_{\Gprj\La}(\La)$ we denote this category by $\CS_{\SG}(\La)$ and call it the $\CG$-monomorphism category of $\La$.
\end{definition}

The following remark collects some of the basic properties of this category.

\begin{remark}\label{SGLa}
Let $\La$ be an Artin algebra.
\begin{enumerate}
   \item The equivalence $\CH(\La) \simeq \mmod T_2(\La)$ in view of \cite[Dual of Theorem 3.5.1]{EHS} or \cite[Theorem 5.1]{LuZ}, implies that $\CS_{\SG}(\La)$ is equivalent to the subcategory of $\mmod T_{2}(\La)$ consisting of all Gorenstein projective $T_{2}(\La)$-modules, that is $\CS_{\SG}(\La) \simeq \Gprj T_{2}(\La)$. In particular, it inherits a Frobenius exact structure from $\Gprj T_{2}(\La)$.
  \item By \cite[Theorem 2.5]{RS}, $\CS(\La)$ as a subcategory of $\CH(\La)$ is functorially finite and hence has Auslander-Reiten sequences. If $\La$ is a selfinjective algebra, then $\CS_{\SG}(\La)$ is exactly $\CS(\La)$. Thus $\CS_{\SG}(\La)$ has almost split sequences over selfinjective algebras.

  \item Let $\Gprj\La$ be a contravariantly finite subcategory of $\mmod\La$.  It is proved in \cite[Proposition 3.1]{H2} that  $\CS_{\SG}(\La)$ is a functorially finite subcategory of $\CH(\La)$. To see this, note that by Item (1) of this remark, $\CS_{\SG}(\La)$ is in fact the subcategory of Gorenstein projective modules in $\CH(\La)$. Hence in view of Remark \ref{Prop. Gprj}(3) it is a resolving subcategory. So by Remark \ref{Rmk: KS}, in order to show that $\CS_{\SG}(\La)$  is a functorially finite subcategory of $\CH(\La)$, it is sufficient to show that it is a contravariantly finite subcategory of $\CH(\La)$. On the other hand, as it is mentioned in Item (2), $\CS(\La)$ is a functorially finite subcategory of  $\CH(\La)$. So in order to show that $\CS_{\SG}(\La)$ is a contravariantly finite subcategory of $\CH(\La)$, it is sufficient to show that every object in $\CS(\La)$ has a right $\CS_{\SG}(\La)$-approximation. This is done explicitly in the proof of \cite[Proposition 3.1]{H2} using facts from basic homological algebra such as Wakamutsu’s Lemma and pull-back diagrams. Therefore, if $\La$ is a Gorenstein algebra or a CM-finite algebra, then $\CS_{\SG}(\La)$ has almost split sequences.
\end{enumerate}
\end{remark}

The following lemma determines the structure of certain almost split sequences in $\CS_{\SG}(\La)$. This lemma   will be referred several times throughout the paper.

\begin{lemma}(\cite[Lemma 6.3]{H})\label{Lemma-ASTM}
Let  $\delta: 0 \rt  A\st{f} \rt B\st{g} \rt C \rt 0$ be an almost split sequence in $\Gprj\La.$
\begin{itemize}
\item[$(1)$] The almost split sequence in $\CS_{\SG}(\La)$ ending at $(0\rt C)$ has the form
\begin{equation}\label{3.2(1)}
\xymatrix@1{  0\ar[r] &
		(A\st{1}\rt A)
			\ar[r]^{( 1 ~~ f)}&
			  (A \st{f}\rt  B) \ar[r]^{( 0 ~~ g )}&
			(0\rt C)\ar[r]& 0. }
\end{equation}	
	
\item [$(2)$] Let $e:A \rt I$ be the minimal left ${\rm prj}\mbox{-}\La$-approximation of $A$. Then the almost split sequence in $\CS_{\SG}(\La)$ ending at $(C\st{1}\rt C)$ has the form
\begin{equation}\label{3.2(2)}
\xymatrix@1{  0\ar[r] &
		(A\st{e}\rt I)
			\ar[rr]^-{( f ~~ \left[\begin{smallmatrix} 1 \\ 0\end{smallmatrix}\right])}
			& & (B \st{h}\rt  I\oplus C) \ar[rr]^-{( g ~~ \left[\begin{smallmatrix} 0 & 1\end{smallmatrix}\right])}& &
			(C\st{1}\rt C)\ar[r]& 0, } \
\end{equation}	
where $e':B \rt I$  is an extension of $e$ and $h$ is the map $\left[\begin{smallmatrix} e' \\ g\end{smallmatrix}\right]$, that is, $h$ is induced by the following push-out diagram
		$$\xymatrix{		 A \ar[d]^e \ar[r]^{f} & B \ar[d]^h
			\ar[r]^{g} & C \ar@{=}[d]  \\
			I \ar[r]^{\left[\begin{smallmatrix} 1 \\ 0\end{smallmatrix}\right]} &I\oplus C\ar[r]^{\left[\begin{smallmatrix} 0 & 1\end{smallmatrix}\right]}
			& C   }	$$
		
\item [$(3)$] Let $b:P\rt C$ be the projective cover of $C$. Then the almost split sequence in $\CS_{\SG}(\La)$ starting  with $(0 \rt A)$ has  the form
\begin{equation}\label{3.2(3)}
\xymatrix@1{  0\ar[r] & (0\rt A)
			\ar[rr]^-{( 0 ~~ \left[\begin{smallmatrix} 1 \\ 0\end{smallmatrix}\right])}
			& & {(\Omega_{\La}(C)\st{h}\rt A\oplus P)}\ar[rr]^-{( 1 ~~ \left[\begin{smallmatrix} 0 & 1\end{smallmatrix}\right])}& &
			{( \Omega_{\La}(C)\st{i}\rt  P)}\ar[r]& 0,}
\end{equation}	
where $b'$ is a lifting of $b$ to $g$ and $h$ is the kernel of the morphism $[f~~b']:A\oplus P\rt B$, that is, $h$ is induced by the following pull-back diagram

$$\xymatrix{		  & \Omega_{\La}(C) \ar[d]^h
			\ar@{=}[r] & \Omega_{\La}(C) \ar[d]^i  \\
			A\ar@{=}[d] \ar[r]^{\left[\begin{smallmatrix} 1 \\ 0\end{smallmatrix}\right]} & A\oplus P
			\ar[d]^{\left[\begin{smallmatrix} f & b'\end{smallmatrix}\right]} \ar[r]^{\left[\begin{smallmatrix} 0 & 1\end{smallmatrix}\right]} & P \ar[d]^b  \\
			A  \ar[r]^{f} & B
			\ar[r]^{g} & C }	$$
\end{itemize}
\end{lemma}

\begin{remark}\label{referee}
As it was mentioned by the referee, using Lemma \ref{Lemma-ASTM}, one can see that three applications of the Auslander-Reiten translation to an object of type $(0 \rt A)$ yields again an object of the same type. This kind of three-periodicity phenomenon, for some of the objects, is happening throughout the paper: for example in Proposition \ref{Proposition 2.5} and  Theorem \ref{Proposition 4.8}.
\end{remark}

Our aim, toward the end of the section, is to study the structure of the middle term of the almost split sequences appearing in Lemma \ref{Lemma-ASTM}.

\begin{setup}
From now on, we assume that $\Gprj\La$ is a contravariantly finite subcategory of $\mmod \La$.  Since $\Gprj\La$ is a resolving subcategory, by Remark \ref{Rmk: KS}, it is a functorially finite subcategory of $\mmod\La$. Hence it has almost split sequences.
\end{setup}

\begin{lemma}\label{Lemma 2.1}
Let $P$ be an indecomposable projective module in $\mmod \La$ with the radical ${\rm rad}P$. Let $i:{\rm rad}P\rt P$ be the canonical inclusion.
\begin{itemize}		
\item [$(1)$] The objects $(0 \rt P)$ and $(P\st{1}\rt P)$   are indecomposable projective-injective objects  in $\CS_{\SG}(\La)$. Furthermore, each indecomposable projective-injective object arises in this way.
\item [$(2)$] The composition map \[(0~~i)\circ(0~~f):(0\rt  G)\rt (0\rt {\rm rad}P)\rt (0\rt P) \]	is  a right minimal almost split morphism in $\CS_{\SG}(\La)$, where $f:G\rt {\rm rad}P$ is the  minimal right $\Gprj\La$-approximation of ${\rm rad}P$.
\item [$(3)$]	The composition map \[(i~~1)\circ (\phi_1~~\phi_2):(G'\st{g}\rt G)\rt ({\rm rad}P\st{i}\rt P)\rt (P\st{1}\rt P) \] is a right minimal almost split morphism in $\CS_{\SG}(\La)$, where $(\phi_1~~\phi_2):(G'\st{g}\rt G)\rt ({\rm rad}P\st{i}\rt P)$ is the minimal right $\CS_{\SG}(\La)$-approximation of $({\rm rad}P\st{i}\rt P)$. Moreover, $G$ belongs to ${\rm prj}\mbox{-}\La.$
\end{itemize}	
\end{lemma}

\begin{proof}
$(1)$  It follows from \cite[Proposition 1]{HM}.
	
$(2)$	Set $K:=\text{Ker}(f)$. By  Wakamutsu's Lemma \cite[Lemma 2.1.1]{X}, $K \in \text{Gprj}\mbox{-}\La^{\perp}$.  For every $(G\st{e}\rt G')$ in $\CS_{\SG}(\La)$, by applying $\Ext^1_{\CH(\La)}(-, (0\rt K) )$ to the short exact sequence
\[0 \rt (G\st{g}\rt {\rm Im}(e))\rt (G\st{e}\rt G')\rt (0\rt \text{Cok}(e))\rt 0,\]
where $g$ is an isomorphism, we obtain that $\Ext^1_{\CH(\La)}((G\st{e}\rt  G'), (0\rt  K) )=0$. Hence $(0\rt K) \in \CS_{\SG}(\La)^{\perp}$. This implies that $(0~~f) :(0\rt G) \rt (	0\rt {\rm rad}P)$ is a right $\CS_{\SG}(\La)$-approximation. Since $f$ is right minimal in $\mmod \La$ we may deduce that $(0~~f)$ is right minimal in $\CS_{\SG}(\La)$. Using this fact, we show that $(0~~if)$ is a right minimal almost split morphism. If $(0~~if)\circ (0~~h)=(0~~if)$, then $hf=f$. Since $f$ is right minimal, $h$ is an isomorphism. Hence $(0~~h)$ is an isomorphism. Thus $(0~~if)$ is right minimal. To complete the proof, we show that it is right almost split. Assume that $(0~~g):(X\st{v}\rt Y)\rt (0\rt P)$ is not a retraction. By \cite[Lemma 3.1]{RS}, the inclusion $(0~~i)$ is a right minimal almost split morphism in $\CS(\La)$. So there exists $(0~~l):(	X\st{v}\rt Y)\rt (0\rt {\rm rad}P)$ such that $(0~~g)=(0~~i)\circ(0~~l)$. Since $(0~~f)$ is a right $\CS_{\SG}(\La)$-approximation, there is a morphism $(0~~d):(X\st{v}\rt Y)\rt (	0\rt G)$ such that  $(0~~l)=(0~~f)\circ (0~~d)$. Hence $(0~~g)$ factors through $(0~~if)$ via $(0~~d)$, as desired.
	
$(3)$ The proof of the first part is similar to the proof of $(2)$. Just note that here we use the fact that $(	{\rm rad}P\st{i}\rt P)\rt (P\st{1}\rt P)$ is a right minimal almost split morphism in $\CS(\La)$ by \cite[Lemma 3.1]{RS}. To prove the second part, let $(K_1\st{d}\rt K_2)$ be the kernel of $(\phi_1~~\phi_2)$. By Wakamutsu's Lemma, $\Ext^1_{\CH(\La)}((X\st{t}\rt Y), (K_1\st{d}\rt K_2))=0$ for any $(X\st{t}\rt Y)$ in $\CS_{\SG}(\La)$. Therefore for any $G \in \text{Gprj}\mbox{-}\La$,
\[\Ext^1_{\CH(\La)}((0\rt G), (K_1\st{d}\rt K_2))=\Ext^1_{\La}(G, K_2)=0,\]
\[\Ext^1_{\CH(\La)}((G\st{1}\rt G), (K_1\st{d}\rt K_2))=\Ext^1_{\La}(G, K_1)=0.\]
Hence we deduce $K_1, K_2 \in \Gprj\La^{\perp}.$ Consider the following short exact sequence in $\CS_{\SG}(\La)$	
\[0 \rt (K_1\st{d}\rt K_2)\rt (G'\st{g}\rt G)\rt ({\rm rad}P\st{i}\rt P)\rt 0.\]
The above sequence gives the short exact sequence $0 \rt K_2\rt G\rt P\rt 0$ in $\mmod \La$, which must split. Hence $K_2 \in \Gprj\La\cap \Gprj\La^{\perp}$. So by Remark \ref{Prop. Gprj}, $G \in \prj\La$. This completes the proof.	
\end{proof}

\begin{lemma}\label{Lemma 2.5}
Let $P$ be a non-zero indecomposable projective module in $\mmod\La.$
\begin{itemize}
\item [$(1)$] Let $( 0 \rt P)$ be a direct summand of the middle term of  an almost split sequence in $\CS_{\SG}(\La)$ starting with $C$. Then $C=(0 \rt G)$, where $G$ is an indecomposable non-projective module in $\Gprj\La.$
\item [$(2)$] Let $( P \st{1}\rt P)$ be a direct summand of the middle term of an almost split sequence  in $\CS_{\SG}(\La)$ starting with $D$. Then $D=( \Omega_{\La}(Q) \st{i}\rt P_Q)$, where $Q$ is  an indecomposable non-projective module in $\Gprj\La.$
\end{itemize}
\end{lemma}

\begin{proof}
By applying Lemma \ref{Lemma 2.1}, we get that a right minimal almost split morphism in $\CS_{\SG}(\La)$ with the codomain $(0 \rt P)$ is of the form $(0 \rt G) \rt ( 0 \rt P)$. This means that the domain of an  irreducible morphism with the codomain $( 0 \rt P)$ is a direct summand of $( 0 \rt G)$.  As a result,  we get the desired form for $C$ in the first statement. The proof of $(2)$ is similar.	
\end{proof}

\begin{proposition}\label{Lemma 2.3}
Consider the same notations as in  Lemma \ref{Lemma-ASTM}(2). If the middle term $(B\st{h}\rt I\oplus C)$ is not indecomposable, then there is  a decomposition
\[(B\st{h}\rt I\oplus C)=(B'\st{h'}\rt C)\oplus (I\st{1}\rt I ),\]
where $(B'\st{h'}\rt C)$ is an indecomposable non-projective object.	
\end{proposition}

\begin{proof}
Assume that the middle term of the sequence \eqref{3.2(2)} is decomposable. We know from the sequence $\delta$ that $C$ is non-projective.
Let $\delta': 0 \rt C\st{f'}\rt B''\st{g'}\rt C'\rt 0$ be the almost split sequence in $\rm{Gprj}\mbox{-} \La$ starting  at $C$. By applying Lemma \ref{Lemma-ASTM}(1, 2) to the almost split sequences $\delta'$ and $\delta$, we get the following almost split sequences in $\CS_{\SG}(\La)$
\[\xymatrix@1{  0\ar[r] & {( C\st{1}\rt C)}
		\ar[rr]^-{(1~~f')}
		& & {(C\st{f'}\rt B'')}\ar[rr]^-{(0~~g')}& &
		{(0\rt C')}\ar[r]& 0 }, \]
\[\xymatrix@1{  0\ar[r] & {(A\st{e}\rt I )}
		\ar[rr]^-{(f~~\left[\begin{smallmatrix} 1 \\ 0\end{smallmatrix}\right])}
		& & {(B\st{h}\rt I\oplus C)}\ar[rr]^-{(g~~\left[\begin{smallmatrix} 0 & 1\end{smallmatrix}\right])}& &
		{(C \st{1}\rt C)}\ar[r]& 0 }, \]		 	
where $h$ is the map $\left[\begin{smallmatrix} e' \\ g\end{smallmatrix}\right]$  with $e':B \rt I$  an extension of $e.$ Since $(C\st{f'}\rt B'')$ is indecomposable and  $\delta'$ is an almost split sequence, the middle term $( B \st{h}\rt I\oplus C)$  can be written as $\mathbb{X}\oplus \mathbb{Y}$, where $\mathbb{X}=( B' \st{h'}\rt D)$ is an indecomposable non-projective object and $\mathbb{Y}$ is a projective, equivalently injective, object in the Frobenius exact category $\CS_{\SG}(\La)$. Using the characterization of projective-injective objects given in Lemma \ref{Lemma 2.1}, we can write $\mathbb{Y}=(I_1\st{1}\rt I_1)\oplus (0 \rt I_2)$ for some projective modules $I_1, I_2$. If there were a non-zero direct summand of $I_2$, say$J$, then by Lemma \ref{Lemma 2.1}, $(A \st{e}\rt I)$ would be a direct summand of $(0 \rt G)$, where $G$ is a minimal right $\text{Gprj}\mbox{-}\La$-approximation of ${\rm rad}J.$ But this means that $A=0$, a contradiction. Hence $(B \st{h}\rt I\oplus C)=(B' \st{h'}\rt D)\oplus (I_1 \st{1}\rt I_1)$. This implies that $B=B'\oplus I_1$ and $I\oplus C=D\oplus I_1$. The second equality along with the fact $C$ is a non-projective module, implies that $C=D$ and $I=I'$. This completes the proof.
\end{proof}

In particular, Proposition \ref{Lemma 2.3} implies that if the middle term of the almost split sequence \eqref{3.2(2)} is decomposable, then there is the following mesh in $\Gamma_{\CS_{\SG}(\La)}$
\[\xymatrix  @R=0.3cm  @C=0.6cm {
			&&&	&&(B'\st{h'}\rt C)\ar[dr]&&&&\\	&&&	&(A\st{e}\rt I)\ar[dr]\ar[ru]\ar@{<.}[rr]\ar[ddr]\ar[dddr]&&(C\st{1}\rt C)&&&\\&&&
			&&(I_1\st{1}\rt I_1)\ar[ru]&&&&
			&&&&&&&&&\\
			&&&&& \vdots\ar[uur]&&&&\\
			&&&& & (I_n\st{1}\rt I_n)\ar[uuur]&&}\]
  where $I=\oplus^n_{i=1}I_i$ is a decomposition of $I$ into a direct sum of indecomposable modules.
	
\begin{proposition}\label{Lemma 2.4}
Consider the same notations as in Lemma \ref{Lemma-ASTM}(3). If the middle term $(\Omega_{\La}(C)\st{h} \rt A\oplus P)$ of the almost split sequence \eqref{3.2(3)} is not indecomposable, then there is a decomposition
\[(\Omega_{\La}(C)\st{h}\rt A\oplus P)=(\Omega_{\La}(C)\st{h'}\rt A)\oplus (0\rt P),\]
where $(\Omega_{\La}(C)\st{h'}\rt A)$ is an indecomposable non-projective object.		
\end{proposition}

\begin{proof}
The proof follows by applying a similar argument as that of Proposition \ref{Lemma 2.3}. Note that here we need to exclude the projective indecomposables of the form $(P\st{1}\rt P)$, and then $(0\rt  A)$  would be a direct summand of  $(G'\st{g} \rt G)$, which is  the minimal right $\CS_{\SG}(\La)$-approximation of $({\rm rad}P\st{i}\rt P)$. Hence it follows from Lemma \ref{Lemma 2.1}(3) that $G$ is projective. So $A$ should be projective as well, which is a contradiction.
\end{proof}

Analogously, in view of Proposition \ref{Lemma 2.4}, we deduce that if the middle term of  the almost split sequence \eqref{3.2(3)} is decomposable, then there is the following mesh in $\Gamma_{\CS_{\SG}(\La)}$
		\[
		\xymatrix  @R=0.3cm  @C=0.6cm {
			&&&	&&(\Omega_{\La}(C)\st{h'}\rt A)\ar[dr]&&&&\\	&&&	&(0\rt A)\ar[dr]\ar[ru]\ar@{<.}[rr]\ar[ddr]\ar[dddr]&&(\Omega_{\La}(C)\st{i}\rt P)&&&\\&&&
			&&(0\rt P_1)\ar[ru]&&&&\\
			&&&&& \vdots\ar[uur]&&&&\\
			&&&& & (0\rt P_n)\ar[uuur]&&			}
		\]
 where $P=\oplus^n_{i=1}P_i$ is a decomposition of $P$	into a direct sum of indecomposable modules.

\section{Stable Auslander-Reiten components }\label{stable componnets}
In this section, we define the notion of boundary vertices in $\Ga^s_{\CS}$, the stable Auslander-Reiten quiver of $\CS_{\SG}(\La)$, and study those components of $\Ga^s_{\CS}$  containing boundary vertices. In particular, when $\La$ is a CM-finite algebra, we provide an explicit description of the shape of such components. Also it will be shown that such components are linked to the orbits of an auto-equivalence on $\underline{\rm Gprj}\mbox{-}\La$.

\begin{setup}
In this section, we assume that $\Gprj\La$ is a contravariantly finite subcategory of $\mmod \La$. Note that, in this case, by Remark \ref{Rmk: KS}, $\Gprj\La$ is a functorially finite subcategory of $\mmod\La$. Hence by Remark \ref{SGLa}(3), $\CS_{\SG}(\La)$ is functorially finite in $\CH(\La)$ and so it has almost split sequences.
\end{setup}

\begin{notation}
Throughout the paper, we let $\vartheta$ stand for the composition $\tau_{\SG}\Omega^{-1}_{\La}\tau^2_{\SG}$, where $\tau_{\SG}$ denotes the Auslander-Reiten translation in $\Gprj\La$. Note that $\vartheta$ is an auto-equivalence on $\underline{\rm Gprj}\mbox{-}\La$.
\end{notation}

\begin{definition}
Let $G$ be  an object in $\underline{\rm Gprj}\mbox{-}\La$. The $\vartheta$-orbit of $G$, denoted by $[G]_{\vartheta}$,  is the set of all modules $\vartheta^iG$ with $i \in \mathbb{Z}$.
\end{definition}

Note that Auslander-Reiten triangles in the triangulated category $\underline{\rm Gprj}\mbox{-}\La$ are induced by almost split sequences in the exact category $\Gprj\La$. It is known that an auto-equivalence on $\underline{\rm Gprj}\mbox{-}\La$ preserves the Auslander-Reiten triangles.

\begin{remark}\label{Remark 4.4}
Since $\Omega^{-1}_{\La}$ is an auto-equivalence on $\underline{\rm Gprj}\mbox{-}\La$, it is easy to see that, with the above notations, the equalities \[\vartheta=\Omega^{-1}_{\La}\tau^3_{\SG}=\tau^3_{\SG}\Omega^{-1}_{\La}\]
hold at the level of objects. These equalities may not be functorial but for our arguments we only need them to be hold true on the objects.
\end{remark}

By Remark \ref{Remark 4.4}, we compute $\vartheta$ for some categories.  Let $d$ be an integer in $\mathbb{N}$. Recall that a Hom-finite triangulated $k$-category $\CT$ with split idempotents is called a $d$-Calabi-Yau category if there is a bifunctorial isomorphism \[\Hom_{\CT}(X, Y)\simeq D\Hom_{\CT}(Y, X[d])\]  for all $X, Y \in \CT.$

\begin{example}
\begin{itemize}
 \item [$(1)$]  Assume that the triangulated category $\underline{\rm{Gprj}}\mbox{-}\La$ is  $d$-Calabi-Yau. Hence there is a functorial isomorphism $\tau_{\SG}=\Omega^{d-1}_{\La}$. Therefore, in this case $\vartheta=\Omega^{3d-4}_{\La}$. Note that for our purpose we only need a weaker condition than the $d$-Calabi-Yau property: we do not need the bifunctorial isomorphism of the definition of a  $d$-Calabi-Yau triangulated category to be compatible with the triangulated structure.
\item [$(2)$] Assume that $\La$ is a finite-dimensional symmetric $k$-algebra over a field $k$. Then by \cite[Corollary IV.8.6]{SY}, $\tau_{\La}=\Omega^2_{\La}$. So, in this case, $\vartheta=\Omega^5_{\La}$.
\item [$(3)$] Assume that $\La$ is a $\CG$-semisimple algebra \cite{H2}; see Definition \ref{Definition 3.13}. Then $\tau_{\SG}=\Omega_{\La}$. Hence, in this case, $\vartheta=\Omega^2_{\La}$.
\end{itemize}
\end{example}

In the following, we compute $\vartheta$-orbits of  a concrete example.

\begin{example}\label{Example 3.3}
Let $\La$ be the $k$-algebra given by the quiver
 		\[
 	\xymatrix  @R=0.3cm  @C=0.6cm {
 		&&&	&&1\ar[dl]_{a}&&&&\\	&&&	&3\ar[rr]_b\ar@{.}[rr]&&2\ar[ul]_{c}&&&}
 	\]
 	and bound by $abc=0, \ bca=0, \ cab=0$. Then $\Gamma^s_{\La}$ is given by
 	\[
 	\xymatrix  @R=0.3cm  @C=0.6cm {&&&&&&&&&&&&& \\
 	  && &(1)\ar[dr]\ar@{<.}[rr]\ar@{.}[u]\ar@{.}[dd]&&(2)\ar[dr]\ar@{<.}[rr]&	&(3)\ar[dr]\ar@{<.}[rr]&&(1)\ar@{.}[u]\ar@{.}[dd]
 		&&&\\ &&& &  _{\left(\begin{smallmatrix}2\\ 1 \end{smallmatrix}\right)}\ar@{.}[l] \ar[ru]\ar@{<.}[rr]&&
 		_{\left(\begin{smallmatrix}3\\ 2 \end{smallmatrix}\right)}\ar@{<.}[rr]\ar[ru]&&_{\left(\begin{smallmatrix}1\\ 3 \end{smallmatrix}\right)}\ar[ru]\ar@{<.}[r]&&&& \\  &&&&&&&&&&} 	\]
where the vertices with the same label are identified, and are described via their composition series. Here $\La$ is selfinjective, so $\tau_{\La}=\tau_{\SG}.$
There are three $\vartheta$-orbits as follows: $\{(1), \left(\begin{smallmatrix}3\\ 2 \end{smallmatrix}\right)\}$, $\{(3), \left(\begin{smallmatrix}2\\ 1 \end{smallmatrix}\right)\}$ and $\{(2), \left(\begin{smallmatrix}1\\ 3 \end{smallmatrix}\right)\}$.
\end{example}

Throughout, for $G \in \Gprj\La$, we denote by $e:G\rt I_{G}$ the minimal left $\prj\La$-approximation of $G$. Moreover, for the ease of notation, we write $\tau_{\CS}$ for $\tau_{\CS_{\SG}(\La)}$.

\begin{proposition}\label{Proposition 2.5}
Assume that $G$ is an indecomposable non-projective Gorenstein projective module, and $d=3m+k$  with $m\geqslant 0, \  0\leqslant k \leqslant 2$. Then the following statements hold.
\begin{itemize}
\item [$(1)$] If $k=0$, then $\tau^{d}_{\CS}(0\rt G)=(0\rt  \vartheta^m G)$ and $\tau^{-d}_{\CS}(0\rt  G)=(0\rt  \vartheta^{-m}G)$.
\item [$(2)$] If $k=1$, then $\tau^{d}_{\CS}(0\rt  G)=(\tau_{\SG}\vartheta^m G\st{1}\rt  \tau_{\SG}\vartheta^m G )$ and  $\tau^{-d}_{\CS}(0\rt G )=(\Omega_{\La}\tau^{-1}_{\SG}\vartheta^{-m}G\st{i}\rt  P_{\tau^{-1}_{\SG}\vartheta^{-m}G} )$.
		\item [$(3)$]If $k=2$, then $\tau^{d}_{\CS}(0\rt G)=(\tau^2_{\SG}\vartheta^m G\st{e}\rt  I_{\tau^2_{\SG}\vartheta^m G})$ and $\tau^{-d}_{\CS}(0\rt G )=(\tau^{-1}_{\SG}\Omega_{\La}\tau^{-1}_{\SG}\vartheta^{-m}G \st{1}\rt \tau^{-1}_{\SG}\Omega_{\La}\tau^{-1}_{\SG}\vartheta^{-m}G)$.
\end{itemize}
\end{proposition}

\begin{proof}
All the statements follow by repeated application of Lemma \ref{Lemma-ASTM}. We  only give a proof  for the first part of $(1)$. We do this by induction on $m$. Assume that  $\tau^{3m}_{\CS}(0\rt G)=(0\rt  \vartheta^m G )$.  Applying Lemma \ref{Lemma-ASTM}(1) to the almost split sequence $0 \rt \tau_{\SG}\vartheta^mG\st{f}\rt B\rt \vartheta^mG\rt 0$ in $\Gprj\La$ gives the following almost split sequence in $\CS_{\SG}(\La)$
\[\xymatrix@1{  0\ar[r] & {(\tau_{\SG}\vartheta^m G\st{1}\rt \tau_{\SG}\vartheta^mG )}
	\ar[r]
	 & {(\tau_{\SG}\vartheta^mG\st{f}\rt B)}\ar[r] &
	{(0\rt  \vartheta^m G)}\ar[r]& 0. } \ \    \]
Hence by the induction hypothesis and the above sequence,  $\tau^{(3m+1)}_{\CS}(0\rt G)=(\tau_{\SG}\vartheta^m G \st{1}\rt \tau_{\SG}\vartheta^m G)$.   Now by applying Lemma \ref{Lemma-ASTM}(2) to the almost split sequence $0 \rt \tau^2_{\SG}\vartheta^mG\rt B'\rt \tau_{\SG}\vartheta^mG\rt 0$ in $\text{Gprj}\mbox{-}\La$, we obtain the following almost split sequence in $\CS_{\SG}(\La)$
\[\xymatrix@1{  0\ar[r] & {( \tau^2_{\SG}\vartheta^mG\st{e}\rt I_{\tau^2_{\SG}\vartheta^mG} )}	\ar[r]
	& {(B'\st{h}\rt I_{\tau^2_{\SG}\vartheta^mG}\oplus \tau_{\SG}\vartheta^mG )}\ar[r] &
	{( \tau_{\SG}\vartheta^m G\st{1}\rt  \tau_{\SG}\vartheta^m G)}\ar[r]& 0. } \ \    \]
Thus $\tau^{(3m+2)}_{\CS}(0\rt G)=(\tau^2_{\SG}\vartheta^m G \st{e}\rt I_{\tau^2_{\SG}\vartheta^m G})$. Finally, by applying Lemma \ref{Lemma-ASTM}(3) to the almost split sequence  $0 \rt \tau_{\SG}\Omega^{-1}_{\La}\tau^2_{\SG}\vartheta^mG\rt B''\rt \Omega^{-1}_{\La}\tau^2_{\SG}\vartheta^mG\rt 0 $ in $\text{Gprj}\mbox{-}\La$, we obtain the following almost split sequence in $\CS_{\SG}(\La)$
{\footnotesize \[\xymatrix@1{  0\ar[r] & {(0 \rt  \tau_{\SG}\Omega^{-1}_{\La}\tau^2_{\SG}\vartheta^mG )}
	\ar[r]
	& {(\tau^2_{\SG}\vartheta^mG \st{h'}\rt \tau_{\SG}\Omega^{-1}_{\La}\tau^2_{\SG}\vartheta^mG\oplus I_{\tau^2_{\SG}\vartheta^mG})}\ar[r] &
	{( \tau^2_{\SG}\vartheta^mG\st{e}\rt  I_{\tau^2_{\SG}\vartheta^mG})}\ar[r]& 0. } \ \    \]}
Note that ${\rm Cok}(e)=\Omega^{-1}_{\La}\tau^2_{\SG}\vartheta^m G.$	 So $\tau^{3(m+1)}_{\CS}(0\rt G )=(0\rt  \tau_{\SG}\Omega^{-1}_{\La}\tau^2_{\SG}\vartheta^m G )=(0\rt  \vartheta^{m+1} G )$. This completes the proof.
\end{proof}

\begin{definition}
Let $M$ be an object in $\CS_{\SG}(\La)$.  The $\tau_{\CS}$-orbit of $M$ is the set of all objects $\tau^n_{\CS}M$, with $n \in \mathbb{Z}$. The module $M$ is called $\tau_{\CS}$-periodic if $\tau^m_{\CS}M \simeq M$ holds for some integer $m\geqslant 1.$
\end{definition}

The following corollary is an immediate consequence of  Proposition \ref{Proposition 2.5}.

\begin{corollary}\label{Corolary 2.6}
Let $G$ be an indecomposable non-projective Gorenstein projective module. If $\La$ is CM-finite, then the indecomposable objects $(0\rt  G)$, $(G\st{1}\rt G )$ and $(\Omega_{\La}(G)\st{i}\rt P_G )$ are $\tau_{\CS}$-periodic.
\end{corollary}

\begin{proof}
We only prove that $(0\rt  G )$ is $\tau_{\CS}$-periodic. One can show the others by using the similar argument. By Proposition \ref{Proposition 2.5}, $\tau^{3m}_{\CS}(0\rt  G)=(0\rt  \vartheta^m G)$ for every $m \in \mathbb{Z}$. Since $\La$ is CM-finite, the set $\{\vartheta^mG \mid m \in \mathbb{Z}\}$ is finite and so the set $\{\tau^{3m}_{\CS}(0\rt  G)\mid m \in \mathbb{Z}\}$ is finite. 	
\end{proof}
We emphasize that $\CS_{\SG}(\La)$ need not  have finite representation type even if $\La$ is CM-finite.

\hspace{ 1 mm}

Let $\Ga^s_{\CS_{\SG}(\La)}$ be the stable Auslander-Reiten quiver of  $\Ga_{\CS_{\SG}(\La)}$. For the ease of notation, $\Ga_{\CS_{\SG}(\La)}$ and $\Ga^s_{\CS_{\SG}(\La)}$ will be denoted by $\Ga_{\CS}$ and $\Ga^s_{\CS}$, respectively.

\begin{definition}\label{Definition-types}
A vertex in $\Ga^s_{\CS}$ is said to be a boundary vertex if it has one of the forms
\[(a) \ (0\rt G); \ \ \ (b) \ (G\st{1}\rt G) \ \ \ {\rm or} \ \ \ (c) \ (\Omega_{\La}(G)\st{i}\rt P),\]
where $G$ is an indecomposable Gorenstein projective module.
\end{definition}

Note that the boundary vertices here are a very substantial generalization of the boundary modules in \cite[(5.1)]{RS2}.

For our next theorem we need to recall the definition of the repetition of a quiver.

\begin{definition} \label{Definition 3.9}
Let $\Delta=(\Delta_0, \Delta_1)$ be a quiver with the sets $\Delta_0$ and $\Delta_1$ of vertices and arrows, respectively. The repetition of $\Delta$, denoted by
$\mathbb{Z}\Delta$, is a quiver which is defined as follows:
\begin{itemize}
\item [$\mbox{-}$] $(\mathbb{Z}\Delta)_0=\mathbb{Z}\times\Delta_0$
\item [$\mbox{-}$] $(\mathbb{Z}\Delta)_1=\mathbb{Z}\times \Delta_1 \cup \sigma(\mathbb{Z}\times \Delta_1)$ with arrows $(n,\alpha):(n,x)\rightarrow (n,y)$ and
  $\sigma(n,\alpha):(n-1,y)\rightarrow (n,x)$ for each arrow $\alpha:x\rightarrow y$ in $\Delta_1$ and $n \in \mathbb{Z}$.
\end{itemize}
\end{definition}

\begin{theorem}\label{Proposition 2.7}
Let  $\La$ be a CM-finite algebra and let $\Ga$ be a component of $\Ga^s_{\CS}$ containing a boundary vertex. Then the following statements hold.
\begin{itemize}
\item [$(1)$] If the cardinality of  $\Ga$ is finite, then $\Ga=\Z\Delta/G$, where $\Delta$ is a Dynkin quiver and $G$ is an automorphism group of $\Z\Delta$ containing a positive power of the translation.
\item [$(2)$] If $\Ga$ is infinite, then  $\Ga$ is a stable tube.
\end{itemize}
In particular, if $\CS_{\SG}(\La)$ is of finite representation type, then $\Ga^s_{\CS}$ is a disjoint union of the finite components containing a boundary vertex.	
\end{theorem}

\begin{proof}
Since we have removed the projective-injective vertices, the component $\Ga$ is stable. According to \cite[Theorem 5.5]{Li},  for both $(1)$ and $(2)$ it suffices to show that the component contains a $\tau_{\CS}$-periodic object. Our assumption in conjunction with Corollary \ref{Corolary 2.6} guarantees the existence of such a vertex. So we only need to consider the case where $\CS_{\SG}(\La)$ is of finite representation type and each component contains a boundary vertex.

Without loss of generality we may assume that $\La$ is an indecomposable algebra, i.e., $\prj\La$ is a connected category. This implies that $\prj T_2(\La)$ is connected as well. Since $\CS_{\SG}(\La)$ contains $\prj T_2(\La)$,  we conclude that  $\CS_{\SG}(\La)$ is a connected category.  Since $\CS_{\SG}(\La)$ is connected and  of finite representation type, it follows from  \cite[Lemma 5.1]{Li} that $\Ga_{\CS}$ is connected. It means that there is a walk between any two vertices. Assume that $\Ga$ is an arbitrary component of  $\Ga^s_{\CS}$. Pick a vertex $x$ of $\Ga.$ The connectedness of $\Ga_{\CS}$ yields a walk $y=x_0\longleftrightarrow x_1 \longleftrightarrow \cdots \longleftrightarrow x_t=x $ such that $y$ is a projective-injective vertex and for $d >0$, $x_d$ is not a projective vertex. Note that $\Ga_{\CS}$ contains all projective-injective vertices. By  $x_d\longleftrightarrow x_{d+1}$ we mean that there is either an arrow $x_d\rt x_{d+1}$ or an arrow $x_{d+1}\rt x_d$ in $\Ga^s_{\CS}$. If $y$ is isomorphic to  $(0\rt  P)$ for some indecomposable projective module $P$, then by Proposition \ref{Lemma 2.4}, $x_1$ has to be of the form $(a)$ or $(c)$.  If $y$ is isomorphic to $(P\st{1}\rt P)$ for some indecomposable projective module $P$, then by Proposition \ref{Lemma 2.3}, $x_1$ has to be of the form $(b)$ or $(c)$. So for  both cases the vertex $x_1$ is boundary, as desired.
\end{proof}

To state our next result we need some notations. Let $G$ be an indecomposable non-projective Gorenstein projective $\La$-module. The unique component of the stable Auslander-Reiten quiver $\Ga^s_{\CS}$ containing the vertex $(0\rt  G)$ will be denoted by $\Ga^s_{\CS}(G)$.

\begin{notation}
Let $U$ be the complete set of all pairwise non-isomorphic indecomposable non-projective Gorenstein projective modules. Set
\[\mathcal{O}:=\{[G]_{\vartheta}\mid G \in U \}, \]
where $\vartheta=\tau_{\SG}\Omega^{-1}_{\La}\tau^2_{\SG}$ to be the set of all $\vartheta$-orbits of the stable category $\underline{\rm{Gprj}}\mbox{-}\Lambda$ and
\[ \mathcal{C}:=\{\Ga^s_{\CS}(G)\mid G \in U\},\]
to be the set of all components of the stable Auslander-Reiten quiver $\Ga^s_{\CS}$ containing a vertex of the form $(0\rt  G)$, where $G \in U$.
\end{notation}

\begin{lemma}\label{Lemma-delta}
There exists a well-defined map $\delta:\mathcal{O}\rt \mathcal{C}$ sending $[G]_{\vartheta}$ to $\Ga^s_{\CS}(G)$.
\end{lemma}

\begin{proof}
Let $G'$ belong to the equivalence class  $[G]_{\vartheta}$. Hence there is an integer $m$ such that $\vartheta^m(G)=G'.$ Proposition \ref{Proposition 2.5} implies that $\tau^{3m}_{\CS}(0\rt  G)=(0\rt  \vartheta^m(G))=(0\rt  G')$. So $(0\rt  G)$ and $(0\rt  G')$ lie in the same $\tau_{\CS}$-orbit, and consequently the same component of $\Ga^s_{\CS}$. Therefore,  the vertices $(0\rt G)$ and $(0\rt  G')$ are connected via a path in $\Ga^s_{\CS}$. By Propositions \ref{Lemma 2.3} and \ref{Lemma 2.4}, the almost split sequences in $\CS_{\SG}(\La)$ with ending terms in the $\tau_{\CS}$-orbits of $(0\rt  G )$  or $(0\rt  G')$ have a non projective-injective direct summand in their middle terms. Hence we can find a path between the vertices $(0\rt  G)$ and $(0\rt  G')$ in $\Ga^s_{\CS}$. But this means that $\Ga^s_{\CS}(G)=\Ga^s_{\CS}(G')$. So $\delta$ is well-defined.
\end{proof}

Let $\mathcal{C}^{\infty}$ denote the subset of $\mathcal{C}$ consisting of all infinite components and $\mathcal{O}^{\infty}$ denote the inverse image of  $\mathcal{C}^{\infty}$ under the map $\delta.$

Recall that the set of all vertices of a stable tube having exactly one immediate predecessor, equivalently having exactly one immediate successor, is called a mouth of the tube.

\begin{proposition}\label{proposition 3.7}
With the above notation, the following statements hold.
\begin{itemize}
\item [$(1)$] The map $\delta$ is surjective.
\item [$(2)$] The restricted map $\delta':\mathcal{O}^{\infty}\rt \mathcal{C}^{\infty}$ is bijective.
\end{itemize}
\end{proposition}

\begin{proof}
The statement $(1)$ follows by definition. To prove the statement $(2)$, assume that the component $\Ga^s_{\CS}(G)$ is infinite, for an indecomposable non-projective Gorenstein projective module $G$. Hence by Proposition \ref{Proposition 2.7}, $\Ga^s_{\CS}(G)$ is a stable tube. By Proposition \ref{Lemma 2.3}, the middle terms of the almost split sequences with ending terms of vertices by the $\tau_{\CS}$-orbit of $(0 \rt G)$ contain exactly one non-projective direct summand. Hence  the $\tau_{\CS}$-orbit of $(0 \rt G)$ generates all the vertices in the mouth of the stable tube.  Since the mouth of any stable tube is unique, we conclude that  $\Ga^s_{\CS}(G)$ is uniquely determined by the equivalence class $[G]_{\vartheta}$. Hence the restricted map $\delta'$ is injective, consequently, it is  bijective by the first part.
\end{proof}

If $\CS_{\SG}(\La)$ is of finite representation type, then it is clear that, for an indecomposable non-projective module $G$ in $\Gprj\La$, the associated component $\Ga^s_{\CS}(G)$ is finite. In contrast, for an indecomposable Artin algebra $\La$ when $\CS_{\SG}(\La)$ is of infinite representation type, we may have both finite or infinite components $\Ga^s_{\CS}(G)$ over an indecomposable algebra $\La$. We see this fact in our next example. Let us preface it with a remark.

\begin{remark}\label{Example-remark}
Let $\Lambda$ be a CM-finite algebra and let  $V$ be the associated stable Cohen-Macaulay Auslander algebra. By \cite[Theroem 6.2]{H} we have the embedding $\Ga_{V} \subseteq \Ga_{\CS}$. It is easy to see that the embedding can be restricted to the embedding $\Ga^s_{V} \subseteq \Ga^s_{\CS}$, and  further under this embedding each component $\Delta$ of $\Ga^s_V$ is contained in exactly one component $\Delta'$ of $\Ga^s_{\CS}$. In addition,  $\Delta$ is finite if and only if $\Delta'$ is finite, see \cite[Theorem 6.1]{H}.
\end{remark}

\begin{example}
Let $\CQ$  be the quiver  $$\xymatrix{ & a  \ar@(l,u)[]^{\alpha} \ar[r]^{\beta} & b },$$ $I$  the ideal generated by $\alpha^3$ and let $A=k\CQ/I$ be the associated bound quiver algebra. Moreover, let $\CQ'$ be the quiver $$\xymatrix{ & c  \ar@(l,u)[]^{\gamma}},$$ $I'$  the ideal generated by $\gamma^6$ and let $B=k\CQ'/I'$ be the associated bound algebra. Assume that $\La$ is a simple gluing algebra of $A$ and $B$, obtained by identifying the vertices $b$ and $c$, see \cite{Lu} for the precise definition of simple gluing of algebras. Then $\La$ is clearly an indecomposable algebra. We claim that the stable Auslander-Reiten quiver $\Ga^s_{\CS_{\SG}(\La)}$ contains a finite component $\Ga^s_{\CS}(G)$ and an infinite component $\Ga^s_{\CS}(G')$ for some indecomposable non-projective modules $G, G'$ in $\text{Gprj}\mbox{-}\La$.

To prove the claim, first note that by \cite[Theorem 4.4]{Lu}, $\underline{\text{Gprj}}\mbox{-}\La\simeq \underline{\text{Gprj}}\mbox{-}A\oplus \underline{\text{Gprj}}\mbox{-}B.$ This equivalence, in turn, implies that $\Ga_{V_3}=\Ga_{V_1} \times \Ga_{V_2}$, where $V_1, V_2$ and $V_3$ are respectively the stable Cohen-Macaulay Auslander algebras of $A$, $B$ and $\La$. Indeed, $\underline{\text{Gprj}}\mbox{-}A\simeq \underline{\text{Gprj}}\mbox{-}k[x]/(x^3)$, as the vertex $b$ is a sink,  and being selfinjective of  $k[x]/(x^i)$ for any $i$, we can consider $V_1$ and $V_2$ as stable Auslander algebras. In view of the classification given in \cite{RS2}, we deduce that every component of $\Ga^s_{\CS(k[x]/(x^3))}$ is finite. So by Remark \ref{Example-remark}, every  component of $\Ga^s_{V_1}$ is also finite. On the other hand, by \cite{RS2} every  component of $\Ga^s_{\CS(k[x]/(x^6))}$ is infinite. So by another use of  Remark \ref{Example-remark}, it  implies that every component of $\Ga^s_{V_2}$  is also infinite. Therefore, we get the indecomposable algebra $\La$ having  the desired property.
\end{example}

 \section{Finite components containing a boundary vertex}\label{Sec 5}

 In this section, we study finite components of $\Ga^s_{\CS}$ containing a boundary vertex. Using the notion of $\CG$-semisimple modules, we classify the components of $\Ga^s_{\CS}$ which consist only of boundary vertices. We then give an example of a component in $\Gamma^s_{\CS}$ with three different $\tau_{\CS}$-orbits each containing a boundary vertex.

\begin{setup}
In this section, we assume that $\Gprj\La$ is a contravariantly finite subcategory of $\mmod \La$. We recall that $\tau_{\SG},$ resp. $\tau_{\CS}$, denotes the Auslander-Reiten translation in $\Gprj\La$, resp. $\CS_{\SG}(\La)$.
\end{setup}

We begin with the following definition.

\begin{definition} \label{Definition 3.13}
Let $G$ be an indecomposable non-projective Gorenstein projective  module. The module $G$ is called $\CG$-semisimple if the canonical short exact sequence $0 \rt \Omega_{\La}(G)\rt P_G\rt G\rt 0$ is an almost split sequence in $\Gprj\La$. Following \cite{H2}, an algebra $\La$ is called $\CG$-semisimple  if every indecomposable non-projective Gorenstein projective module in $\Gprj\La$ is $\CG$-semisimple.
\end{definition}

\begin{lemma}\label{Lemma 2.9}
Let $G$ be a $\CG$-semisimple  module. Then for any $i \in  \Z$, $\Omega^i_{\La}(G)$ is a $\CG$-semisimple module.  In particular, for any $i \in \Z$, $\tau^i_{\SG}G=\Omega^i_{\La}(G)$. 	
\end{lemma}

\begin{proof}
We only prove the lemma for positive integers $i$. We do this by induction. The statement holds true for $i=0$ by assumption. So we assume that $i> 0$. Consider  the almost split sequence $0 \rt \tau_{\SG}\Omega^i_{\La}(G)\rt B\rt \Omega^i_{\La}(G)\rt 0$ in $\Gprj\La$. We claim that $B$ has no  non-projective indecomposable direct summand. Assume to contrary that $B$ has a non-projective indecomposable direct summand, say $C$. This implies that $P^{i-1}$ has the non-projective summand $\tau^{-1}_{\SG}C$, because by induction assumption, $P^{i-1}$ is the middle term of the almost split sequence in $\Gprj\La$ starting with $\Omega^i_{\La}(G)$. This contradicts the fact that $P^{i-1}$ is projective. Hence $B$ is a projective module. Consequently, by using the fact that the morphisms involved in an almost split sequence are minimal, we deduce that	$\tau_{\SG}\Omega^i_{\La}(G)=\Omega^{i+1}_{\La}(G)$.
\end{proof}

\begin{proposition}\label{G-semisimpeVertex}
Let $\La$ be a CM-finite algebra and let $G$ be a $\CG$-semisimple module. Then the irreducible morphisms in the  component $\Ga^s_{\CS}(G)$ define  an oriented cycle
\[
\xymatrix@C=-1em@R=3ex{ & (0\rt G) \ar[dl]\\
(G\st{1}\rt G)\ar[dd] & & (\Omega_{\La}(G) \st{u_0}\rt P^0)\ar[ul] \\
\\
(G\st{u_{n-1}}\rt P^{n-1})\ar[dr] & & (\Omega_{\La}(G)\st{1}\rt \Omega_{\La}(G))\ar[uu] \\
& \cdots\ar[ur] }
\]
and which consists of the vertices
\[( 0\rt  \Omega^i_{\La}(G)), \ ( \Omega^i_{\La}(G)\st{1}\rt \Omega^i_{\La}(G)) \ \text{and} \ \  (\Omega^{i+1}_{\La}(G)\st{u_i}\rt P^i), \]
where $0 \leqslant i \leqslant n-1$ and $u_i$ denotes the canonical inclusion.
\end{proposition}

\begin{proof}
Since $\La$ is CM-finite, the set $\{\Omega^i_{\La}(G)\mid i \in \Z \}$ is finite. Hence we may choose a minimal positive number $n$ with $G=\Omega^n_{\La}(G)$. Consider the following exact sequence induced by a minimal projective resolution of $G$
\[0 \rt \Omega^n_{\La}(G)\rt P^{n-1}\rt \cdots P^1\rt P^0\rt G\rt 0.\]

We split the above sequence to the following short exact sequences
\[\epsilon_i:  0 \rt \Omega^{i+1}_{\La}(G)\rt P^i \rt \Omega^i_{\La}(G)\rt 0,\]
where $0 \leqslant i \leqslant n-1$. By Lemma \ref{Lemma 2.9}, all sequences $\epsilon_i$ are almost split sequences in $\Gprj\La$.
	
By applying Propositions \ref{Lemma 2.3}, \ref{Lemma 2.4} and Lemma \ref{Lemma-ASTM}(1) to the short exact sequences $\epsilon_0$ and $\epsilon_1$, we obtain the following full subquiver of $\Ga^s_{\CS}$.
{\tiny 
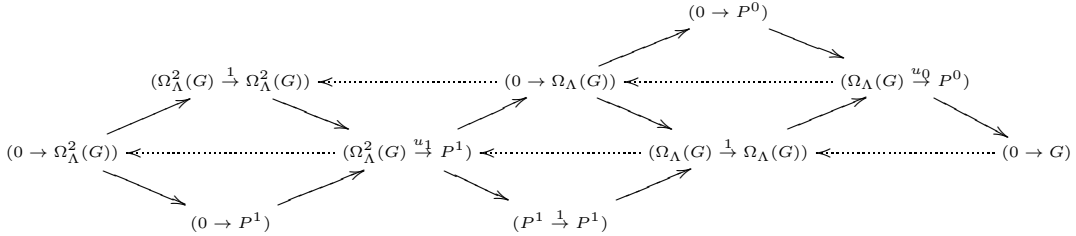
{\[
\xymatrix  @R=0.4cm  @C=0.2cm {
&&&&(0\rt P^0)\ar[dr]& &&	\\   &(\Omega^2_{\La}(G)\st{1}\rt \Omega^2_{\La}(G))\ar[dr]\ar@{<.}[rr]&	&(0\rt \Omega_{\La}(G))\ar[dr]\ar[ru]\ar@{<.}[rr]&&
	(\Omega_{\La}(G)\st{u_0}\rt P^0)\ar[dr]&&&\\  (0\rt \Omega^2_{\La}(G))\ar[dr]\ar[ru]\ar@{<.}[rr]&&
	(\Omega^2_{\La}(G)\st{u_1}\rt P^1)\ar[dr]\ar@{<.}[rr]\ar[ru]&&(\Omega_{\La}(G)\st{1}\rt \Omega_{\La}(G))\ar[ru]\ar@{<.}[rr]&&(0\rt G)&&\\
		&(0\rt P^1)\ar[ur]&	&(P^1\st{1}\rt P^1)\ar[ru]&&&&&}
\]}}
Note that projective modules $P^0$ or $P^1$ might not be indecomposable, which is not a problem in our argument. Repeating the same construction for the pairs of the short exact sequences $(\epsilon_1, \epsilon_2)$, $(\epsilon_2, \epsilon_3)$, $\ldots$, $(\epsilon_{n-2}, \epsilon_{n-1})$, we obtain $n-1$ full subquivers of $\Ga^s_{\CS}$ as the above such that one corresponding to $(\epsilon_{n-2}, \epsilon_{n-1})$ has the object $(0 \rt G )$  in the leftmost side. Hence the construction will stop at $(n-1)$-th step. By gluing the obtained full subquivers we get the full subquiver $\tilde{\Ga}$ of $\Ga^s_{\CS}$ containing the $\tau_{\CS}$-orbit of $(0 \rt  G)$.    By deleting the projective-injective vertices of  the full subquiver $\tilde{\Ga}$,    we then get the following component $\tilde{\Ga}^s$ of $\Ga^s_{\CS}$ containing the vertex $( 0\rt  G)$
{\tiny{\[
	\xymatrix  @R=0.5cm  @C=0.7mm {
		&(G\st{1}\rt G)\ar[rd]\ar@{<.}[r]&&&\cdots&	&(0\rt \Omega_{\La}(G))\ar[dr]\ar@{<.}[rr]&&(\Omega_{\La}(G)\st{u_0}\rt P^0)\ar[dr]&\\(0\rt G)\ar[ru]\ar@{<.}[rr]&&(G\st{u_{n-1}}\rt P^{n-1})&&\cdots&
		(\Omega^2_{\La}(G)\st{u_1}\rt P^1)\ar@{<.}[rr]\ar[ru]&&(\Omega_{\La}(G)\st{1}\rt \Omega_{\La}(G))\ar[ru]\ar@{<.}[rr]&&(0\rt G)}\]}}
which releases all facts of this proposition.
\end{proof}

\begin{remark}
 The proof of Proposition \ref{G-semisimpeVertex} implies that the cardinality of $\Ga^s_{\CS}(G)$ is equal to $3n$, where $n$ is the least positive integer satisfying $\Omega^n_{\La}(G)=G$. In the next section we will investigate this three-periodic phenomenon in a more general setting.
\end{remark}

The following lemma is needed in the proof of our next theorem.

\begin{lemma}\label{Lemma-SY}
 Let $X$ and $ Y$ be  two vertices in the same connected  component of  $\Ga^s_{\CS}$ lying in different $\tau_{\CS}$-orbits. Then there is a sectional path in $\Ga^s_{\CS}$ from $X$ to $\tau^m_{\CS}Y$, for some integer $m$.
\end{lemma}

\begin{proof}
The same argument as in the proof of  \cite [Lemma IV.15.5]{SY} applies to this setting too.
\end{proof}

\begin{theorem}\label{Proposition 2.13}
Let $G$ be an indecomposable non-projective Gorenstein projective module. If the associated component $\Ga^s_{\CS}(G)$ is finite, then $\Ga^s_{\CS}(G)$ contains  at most two distinct $\tau_{\CS}$-orbits containing a boundary vertex different from the $\tau_{\CS}$-orbit of the boundary vertex $(0 \rt G)$.
\end{theorem}

\begin{proof}
Assume that
\[X_n\st{a_n}\rt \cdots \rt X_2\st{a_2}\rt X_1=(0 \rt G)\]
is the longest sectional path in $\Ga^s(G)$. 	According to Proposition \ref{Proposition 2.7},  we have  $\Ga^s_{\CS}(G)=\Z\Delta/G$ for some valued Dynkin quiver $\Delta.$ Therefore, the following three forms of the meshes could appear in $\Ga^s_{\CS}(G)$
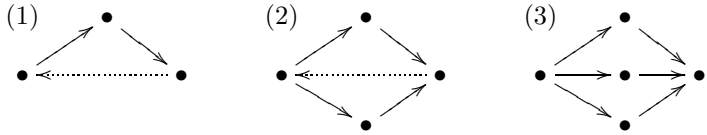
\[\xymatrix  @R=0.3cm  @C=0.6cm {
	(1)	&\bullet\ar[dr]&&\\		\bullet\ar[ru]\ar@{<.}[rr]&&\bullet&}  \xymatrix  @R=0.3cm  @C=0.6cm {
(2)	&\bullet\ar[dr]&&\\	\bullet\ar[dr]\ar[ru]\ar@{<.}[rr]&&\bullet&\\
	&\bullet\ar[ru]&&}  \xymatrix  @R=0.3cm  @C=0.6cm {
(3)	&\bullet\ar[dr]&&\\		\bullet\ar[dr]\ar[ru]\ar[r]&\bullet\ar[r]&\bullet&\\
	&\bullet\ar[ru]&&}. \]
Since $X_1$ is a boundary vertex and $X_n$ is the starting vertex of the longest sectional path, the meshes containing  $X_1, \tau_{\CS}X_1$ and $X_n, \tau_{\CS}X_n$ are of the form $(1)$. On the other hand, for each $ 2 \leqslant i \leqslant n-1$, the mesh $\Xi_i$ containing $X_i, \tau_{\CS}X_i$ is either of the form $(2)$ or form $(3)$. More precisely, when $\Delta$ has an underlying graph of one of the Dynkin types $\mathbb{A}, \mathbb{B}, \mathbb{C}, \mathbb{F}$ and $\mathbb{G}$, then all $\Xi_i$  are of the form  $(2)$, but for the remaining types there exists exactly one  $ 1< j< n$ such that the mesh $\Xi_j$ containing $X_j, \tau_{\CS}X_j$ is of the form  $(3)$. In this case,
let $Y_j$ be the middle vertex of the mesh $\Xi_j$ other than $X_{j+1}$ and $\tau_{\CS}X_{j-1},$  then the mesh containing $Y_j$ as the leftmost vertex must be   of the form $(1)$. These facts follow easily from the covering of valued translation quiver $\Z\Delta\rt \Z\Delta/G.$ By using these facts we observe that all the sectional paths ending at $X_1$ are as follows:
$$2\leqslant i \leqslant n \ \ \ \ u_i: X_i\st{a_i}\rt X_{i-1}\rt \cdots \st{a_2  }\rt X_1 ,$$
and
$$\ \ \   \ \ \ \ \  \ \ \ \  \ \  \ \ v_j: Y_j\st{b_j}\rt X_j\st{a_j}\rt  \cdots\st{a_2}\rt  X_1, $$
where $Y_j$, if it exists, is the middle term of the mesh $\Xi_j$ different from $X_{j+1}$ and $\tau_{\CS}X_{j-1}.$  Following this way, we conclude that, for every $m$, the sectional path ending at $\tau^m_{\CS}X_1$, is obtained by applying the $\tau^m_{\CS}$ on $u_i$ or $v_j$, if it  exists.

This fact, in view of Lemma \ref{Lemma-SY}, implies that any vertex of $\Ga^s_{\CS}(G)$ lies in the $\tau_{\CS}$-orbit of $X_i$, for $1 \leq i \leq n$, or $Y_j$, if such a $j$ exists. Therefore, only possible vertices lying as the leftmost or rightmost vertex of  a mesh with exactly one vertex in the middle are the vertices belonging to the $\tau_{\CS}$-orbits of  $X_1, X_{n}$ and $Y_j$.  Hence we get our result, because boundary vertices satisfying such a property on the middle vertices.
\end{proof}

In particular, Theorem \ref{Proposition 2.13} implies that the inverse image of every single subset of $\mathcal{C}$  under the map $\delta$, defined in Lemma \ref{Lemma-delta}, has  cardinality at most 3. In the next example, we see that a finite component $\Ga^s_{\CS}(G)$ may contain  three different $\tau_{\CS}$-orbits of boundary vertices. For that, we need the construction of $\CS_{\SG}(T_2(\La))$, the $\CG$-monomorphism category of $T_2(\La)$. Let us review briefly its construction.

\begin{remark} The objects of the homomorphism category $\CH(T_2(\La))$ are the commutative diagrams in $\mmod\La$
\begin{equation}\label{diagram5.2}	 \xymatrix{ B\ar[r]^{\alpha_2} & D\\
	A	\ar[u]^{f} \ar[r]^{\alpha_1} & C. \ar[u]_{g}}
\end{equation}		
 A  morphism from an object as above to another object

\[ \xymatrix{ B'\ar[r]^{\alpha'_2} & D'\\
	A'	\ar[u]^{f'} \ar[r]^{\alpha'_1} & C' \ar[u]_{g'}}\]
is a 4-tuple $(\phi_1, \phi_2, \phi_3, \phi_4)$ of morphisms in $\mmod \La$ satisfying the following commutative diagram

\[\xymatrix@C-1.pc@R-1pc{ & B' \ar[rr]^{\ \al'_2} \ar@{<-}[dl]_{\phi_2} \ar@{<-}[dd]^>>>>{f'} && D' \ar@{<-}[dl]^{\phi_4} \ar@{<-}[dd]^>>>>>{g'}\\
B\ar@{<-}[dd]^>>>>>>>>f \ar@{->}'[r][rr]^{   \al_2} && D\ar@{<-}[dd]^>>>>>>>>{g}  &\\
& A' \ar[rr]^{\alpha'_1 \ \ \ \ \ } \ar@{<-}[dl]_>>>>>{\phi_1}  && C'\ar@{<-}[dl]^{\phi_3} \\
A \ar@{->}[rr]^{\alpha_1} && C&
}\]

We will adhere to the convention that objects of $T_2(\La)$ will be depicted vertically and morphisms will be depicted horizontally. Now in view of the fact that $\Gprj T_{2}(\La)  \simeq \CS_{\SG}(\La)$, the objects of $\CS_{\SG}(T_2(\La))$ are all such commutative diagrams \ref{diagram5.2} with the property that objects $f$ and $g$ lie in $\CS_{\SG}(\La)$, the morphisms $\alpha_1$ and $\alpha_2$ are monomorphisms and the induced morphism $({\rm Cok}(\alpha_1) \st{h}\rt {\rm Cok}(\alpha_2))$ lies in $\CS_{\SG}(\La)$.

The objects of $\CS_{\SG}(T_2(\La))$ are certain commutative diagrams in $\mmod \La$ or certain morphisms in ${\CH}(T_2(\La))$.  For purpose of simplicity, we follow the second description.
\end{remark}
\begin{example}\label{Example 4.6}
The triangular matrix algebra $T_2(k[x]/(x^2))$ is a $k$-algebra given by the quiver
\[\xymatrix{\underset 1 {\bullet}
	\ar@(ul,ur)[]^{\alpha_1}\ar[r]^\beta &\underset 2 \bullet
	\ar@(ul,ur)[]^{\alpha_2} }\]
with relations
$\alpha_1^2, \ \alpha_2^2,  \ \alpha_1\beta -\beta\alpha_2$.
First, we draw the Auslander-Reiten quiver of $\CS(k[x]/(x^2))$. Remember that $\CS_{\SG}(k[x]/(x^2))\simeq \Gprj T_2(k[x]/(x^2))$.
\[
\xymatrix  @R=0.3cm  @C=0.6cm {
	&&&&&& &&\\	&&&&&(0\rt \La)\ar[dr]& &&	\\   &&&	&(0\rt S)\ar[dr]\ar[ru]\ar@{<.}[rr]\ar@{.}[l]&&
	(S\st{i}\rt \La)\ar@{.}[uu]\ar@{.}[dd]&&&\\  &&&
	(S\st{i}\rt \La)\ar[dr]\ar@{<.}[rr]\ar[ru]\ar@{.}[uuu]\ar@{.}[d]&&(S\st{1}\rt S)\ar[ru]\ar@{<.}[r]&&&&\\
	&&&	&(\La\st{1}\rt \La)\ar[ru]&&&&&}
\]
\noindent
where $\La=k[x]/(x^2)$ and  $S=k[x]/(x)$. To compute the above Auslander-Reiten quiver, we apply  the fact that the canonical short exact sequence  $0 \rt S\st{i}\rt \La\st{\pi}\rt S\rt 0$ is an almost split sequence in $\Gprj\La.$ The vertices with the same label are identified.

The above Auslander-Reiten quiver implies that the stable Auslander algebra $\Delta$ of $T_2(\La)$ is the following Nakayama algebra of Loewy length 2
\[\xymatrix  @R=0.3cm  @C=0.6cm {
	&(S\st{i}\rt \La)\ar[dr]&&\\		(S\st{1}\rt S)\ar[ru]&&(0\rt S),\ar[ll]&}\]
which is the path algebra of the above quiver modulo the square of the arrow ideal.

Set  $G:=( 0\rt  S)$, $H:=(S\st{1}\rt  S)$ and $K:=(S\st{i}\rt \La)$ as the objects of $\CS_{\SG}(\La)$.

The above Auslander-Reiten quiver, in view of the diagram,
\begin{equation} \label{diagramexample}
 	\xymatrix{0 \ar[r] & S \ar[r]^{i} \ar[d]^{i} & \La \ar[r]^{\pi} \ar[d]^{\left[\begin{smallmatrix} 1\\ 0\end{smallmatrix}\right]} & S  \ar[r] \ar[d]^{i} & 0 \\
		0 \ar[r] & \La \ar[r]_<<<<{\left[\begin{smallmatrix} 1\\ -i\pi \end{smallmatrix}\right]} & \La\oplus \La \ar[r]_{[i \pi~~1]} & \La  \ar[r]  & 0.}
		\end{equation}
allows us to compute the $\tau_{\CS}$-orbits in $\CS_{\SG}(T_2(\La))$:
$$\tau_{\CS}(0\rt  G)=( H\st{1}\rt  H), \ \ \ \tau^{2}_{\CS}(0\rt G)=( K\st{i_1}\rt P_K), \ \ \  \tau^{3}_{\CS}( 0\rt  G)=(0\rt  G); $$
$$\tau_{\CS}(0\rt  H)=( K\st{1}\rt  K), \ \ \ \tau^{2}_{\CS}( 0\rt H)=(G\st{i_2}\rt P_G), \ \ \  \tau^{3}_{\CS}( 0\rt H)=( 0\rt  H); $$
$$\tau_{\CS}( 0\rt  K)=(G\st{1}\rt G), \ \ \ \tau^{2}_{\CS}( 0\rt  K)=(H\st{i_3}\rt P_H), \ \ \  \tau^{3}_{\CS}( 0\rt  K)=(0\rt  K). $$
 Note that here $K=\Omega_{T_2(\La)}(K)$.

By these computations, we have three different $\tau_{\CS}$-orbits of the boundary vertices in $\CS_{\SG}(T_2(\La))$. Since $\Delta$ is of finite representation type, its  Auslander-Reiten quiver $\Ga_{\Delta}$ has only one component $\Ga_{\Delta}\subseteq \Ga_{\CS_{\SG}(T_2(\La))}$, as we discussed in Remark \ref{Example-remark}. Under this embedding we can identify the vertices of $\Ga_{\Delta}$ by the vertices of $\CS_{\SG}(T_2(\La))$. In particular, the projective vertices of $\Ga_{\Delta}$  are identified with the boundary vertices of type $(c)$ in $\Ga_{\CS_{\SG}(T_2(\La))}$.  The quiver $\Ga_{\Delta}$ is also contained in $\Ga^s_{\CS_{\SG}(T_2(\La))}$ as we remove only projective-injective vertices of $\Ga_{\CS_{\SG}(T_2(\La))}$. Then we observe that  $\Ga_{\Delta}$ contains the vertices $ ( K\st{i_1}\rt P_K), (G\st{i_2}\rt P_G)$ and $(H\st{i_3}\rt P_H)$. This implies that $\Ga^s_{\CS}(G)=\Ga^s_{\CS}(H)=\Ga^s_{\CS}(K)$.

We refer the reader to \cite[Example 6.5]{H} for more detailed information on the Auslander-Reiten quivers $\Ga_{\Delta}, \Ga_{\CS_{\SG}(T_2(\La))}$ and   the embedding $\Ga_{\Delta} \subseteq \Ga_{\CS_{\SG}(T_2(\La))}$. Under this embedding, the vertices of $\Gamma_{\Delta}$  correspond to the vertices of $\Ga_{\CS_{\SG}(T_2(\La))}$ not being of the  types $(a)$ or $(b)$, see Definition \ref{Definition-types}. However, for the convenience of the reader, we provide the relevant Auslander-Reiten quivers according to the notations used in the current paper. Set $P:=(0 \rt \La)$ and $Q:=(\La\st{1}\rt \La)$, where $\La=k[x]/(x^2)$.
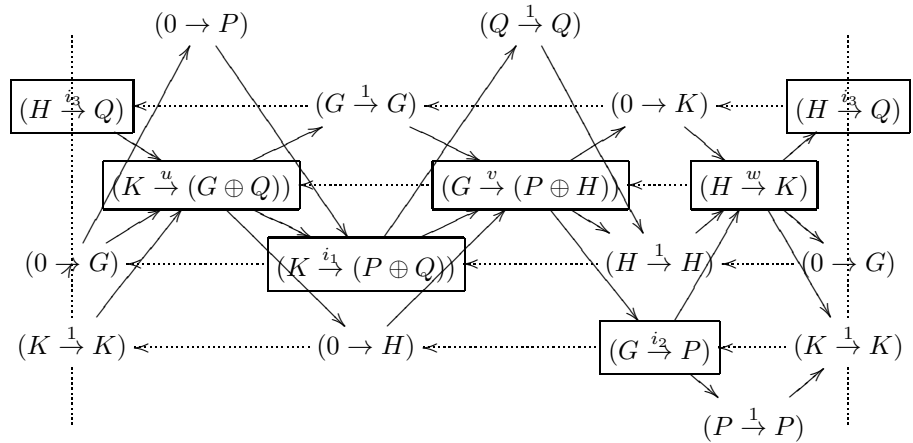
\[\xymatrix@-3.0pc@R=10pt{ &&&& {\tiny{(0\rt P)}}\ar[dddr]& &{\tiny{(Q\st{1}\rt Q)}}\ar[dddr] & &
 		& & & & \\&&
 		&*+[F]{\tiny {(H\st{i_3}\rt Q)}}\ar[dr] & &
 		{\tiny{(G\st{1}\rt G)}}\ar[dr]\ar@{.>}[ll] & &
 		{\tiny{(0\rt K)}}\ar[dr]\ar@{.>}[ll]
 		& &
 	*+[F]{{(H\st{i_3}\rt Q)}}\ar@{.>}[ll]	& &
 		&	& \\
 		&& 	&& *+[F]{\tiny{(K\st{u}\rt (G\oplus Q))}}\ar[dr]\ar[ur]\ar[ddr] & &
 		*+[F]{\tiny{(G\st{v}\rt (P\oplus H))}}\ar[ddr]\ar[dr]\ar[ur]\ar@{.>}[ll] & &
 		*+[F]{{\tiny (H\st{w}\rt K)}}\ar[ddr]\ar[dr]\ar@{.>}[ll]\ar[ur]
 		& & & &
 		\\&&& {(0\rt G)}\ar[ur]\ar[uuur]\ar
 		&&
 		*+[F]{\tiny{(K\st{i_1}\rt (P\oplus Q))}}\ar[ur]\ar[uuur]\ar@{.>}[ll] & &
 		{\tiny{(H\st{1}\rt H)}}\ar[ur]\ar@{.>}[ll] & &{\tiny{(0\rt G)}}\ar@{.>}[ll]
 		& &
 		& \\&& & {(K\st{1}\rt K)}\ar[uur]\ar@{.}[d]\ar@{.}[uuuu]
 		& & {\tiny{(0\rt H)}}\ar[uur]\ar@{.>}[ll]&
 		&*+[F]{{\tiny{(G\st{i_2}\rt P)}}}\ar[uur]\ar[dr]\ar@{.>}[ll] &
 		&{\tiny{(K\st{1}\rt K)}}\ar@{.>}[ll]\ar@{.}[d]\ar@{.}[uuuu] &
 		& &
 		\\&& &
 		& &&		
 		& & {\tiny{(P\st{1}\rt P)}}\ar[ur]  & &		
 		& & }\] 		
The vertices of $\Delta$ are displayed by the solid frames. One can see easily that the projective covers $P_G$ and $P_H$ are $P$ and $Q$, respectively. Moreover, the commutative diagram \ref{diagramexample} implies that $P_K=P\oplus Q$.
\end{example}

\section{Three-periodicity phenomenon}\label{Sec 6}
In this section, we investigate an interesting application of Proposition \ref{Proposition 2.5}, which provides some information concerning the cardinality of the finite components of the stable Auslander-Reiten quiver $\Ga^s_{\CS}$ containing a boundary vertex. We are grateful to the referee for drawing our attention to this application.

In this section, let $\La$ be a finite dimensional algebra over an algebraically closed field. Let $\Ga$ be a component of $\Ga^s_{\CS}$ containing a boundary vertex and with a finite cardinality. It follows from  Proposition \ref{Proposition 2.7} that  $\Ga=\Z\Delta/G$, where $\Delta$ is a Dynkin quiver and $G$ is an admissible automorphism group of $\Z\Delta$ containing a positive power of the translation. Assume that $\Ga^s_{\CS}$ is a finite quiver.  Since $\Ga^s_{\CS}$ is the Auslander-Reiten quiver of the triangulated category $\underline {\CS_{\SG}(\La)}$, we obtain by \cite{Am} or \cite{XZ} that the group $G$ is a  weakly
admissible automorphism group,  which is  isomorphic to $\Z$. Riedtmann \cite[Section 4.3]{Rie} described all possible generators. One can find  a complete list of all possible generators in \cite[Theorem 2.2.1]{Am}.

Let us fix a numbering and an orientation of the simply-laced Dynkin trees.

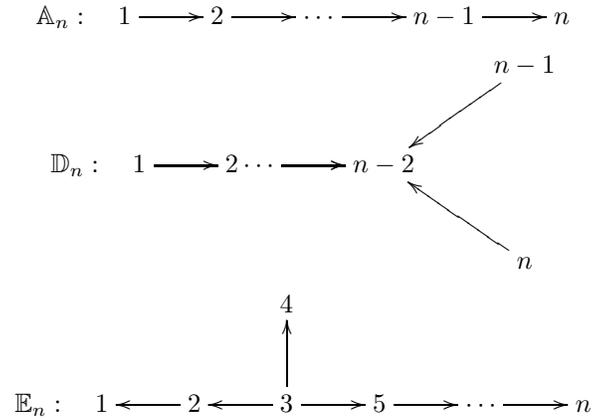
\begin{displaymath}
\xymatrix{\mathbb{A}_n:\quad 1\ar[r] & 2\ar[r] & \cdots \ar[r]& n-1\ar[r] & n}
\end{displaymath}
$$\xymatrix{& & &  n-1\ar[dl]\\ \mathbb{D}_n:\quad 1\ar[r] & 2\ar[r] \cdots\ar[r] & n-2 & \\ & & & n\ar[ul]}$$
$$\xymatrix{ & & 4 & & & \\ \mathbb{E}_n:\quad 1 & 2\ar[l] &
  3\ar[l]\ar[r]\ar[u] & 5\ar[r] & \cdots \ar[r] & n}$$

The quiver $\Z\Delta$ with the translation $\tau(n,x)=(n-1,x)$ is clearly a stable translation quiver which does not depend (up to isomorphism) on the
orientation of $\Delta$, see \cite{Rie}.
Let $\Delta=\mathbb{A}_n$, and define an automorphism $S$ of $\Z\Delta$ by sending $(p, q)$ to $(p+q, n+1-q)$.

\begin{theorem}\label{groupesadmissibles}(\cite[Theorem 2.2.1]{Am})
Let $\Delta$ be a Dynkin tree quiver and $G$ be a non-trivial group of weakly admissible automorphisms of $\Z\Delta$. Then $G$ is isomorphic to $\Z$ and here is a list of its possible generators:
\begin{itemize}
\item
if $\Delta=\mathbb{A}_n$ with $n$ odd, then possible generators are
$\tau^r$ and $\phi\tau^r$ with $r\geq 1$, where $\phi=\tau^\frac{n+1}{2} {S}$
is an automorphism of $\Z\Delta$ of order 2.
\item
if $\Delta=\mathbb{A}_n$ with $n$ even, then possible generators are $\rho^r$, where $r\geq 1$ and
where $\rho=\tau^{\frac{n}{2}} {S}$. Since $\rho^2=\tau^{-1}$, $\tau^r$ is a possible generator.
\item
if $\Delta=\mathbb{D}_n$ with $n\geq 5$, then possible generators are
$\tau^r$ and $\tau^r\phi$, where $r\geq 1$ and
where $\phi= (n-1,n)$ is the automorphism of $\mathbb{D}_n$ exchanging $n$
and $n-1$.
\item
if $\Delta=\mathbb{D}_4$, then possible generators are
$\phi\tau^r$, where $r\geq 1$ and where $\phi$ belongs to
$\mathfrak{S}_3$ the permutation group on 3 elements seen as a subgroup of
automorphisms of $\mathbb{D}_4$.
\item
if $\Delta=\mathbb{E}_6$, then possible generators are
$\tau^r$ and $\phi\tau^r$, where $r\geq 1$ and where $\phi$ is the
automorphism of $\mathbb{E}_6$ exchanging $2$ and $5$, and $1$ and $6$.
\item
if $\Delta=\mathbb{E}_n$ with $n=7,8$, then possible generators are $\tau^r$, where $r\geq 1$.
\end{itemize}

\end{theorem}

Now we are ready to prove our main theorem in this section.

\begin{theorem}\label{Proposition 4.8}
Let   $\La$ be a finite dimensional algebra over an algebraically closed field and  let $\CS_{\SG}(\La)$ be of finite representation type. Let $\Delta$ be a Dynkin diagram different from $\mathbb{D}_4$ and let $G$ be a non-trivial group of weakly admissible automorphisms of $\Z\Delta$. If $\Gamma=\mathbb{Z}\Delta/G$ is a component of the stable Auslander-Reiten quiver $\Gamma^s_{\CS}$,  then 3 is a divisor of the cardinality of $\Ga$.
\end{theorem}

\begin{proof}
Assume that  $X$ is a non-projective indecomposable Gorenstein-projective module. Then Proposition \ref{Proposition 2.5} implies that three applications of the Auslander-Reiten translation to the object ${\bf X}=(0 \rt X)$ yield again an object of the same type. Hence the cardinality $d$ of the $\tau_{\CS}$-orbit of ${\bf X}$ is a divisor of 3. In particular, ${\bf X }$ is of $\tau_{\CS}$-period $d$. Let us fix the following notation throughout the proof.
\begin{itemize}
\item Let $\alpha$ denote the possible generator of $G$.
\item For every $i \in \Z$ and a vertex $a$ in $\Delta$,  we let $\widetilde{(i, a)}$ denote  the associated  $\alpha$-orbit of the vertex $(i, a)$ in $\Z\Delta$, that is, the set of all $\alpha^j(i, a)$ where $j$ runs through $\Z$ .
\item Let ${\bf X}$ be a vertex of $\Ga=\Z\Delta/G$ that corresponds to the orbit of a vertex $(0, x)$ in $\Z\Delta$, where $x$ is a vertex in $\Delta.$ Namely, ${\bf X}= \widetilde{(0, x)}$.
\item Let $U(v)$ denote the $\tau_{\CS}$-orbit of a vertex $v$ in $\Ga^s_{\CS}$.
\end{itemize}

We check all possible cases based on Theorem \ref{groupesadmissibles} and show that for each case    $3$ is a divisor of $\mid \Gamma \mid$.

{(1)} $\Delta=\mathbb{E}_n$ with $n=7, 8.$  By Theorem \ref{groupesadmissibles}, we have $\alpha=\tau^r$ for some $r\geq 1.$ According to the existence  of the boundary vertex ${\bf X}$ with $\tau^d_{\CS} \widetilde{(0, x)}=\widetilde{(0, x)}$ in $\Ga$,  we deduce that $(-d, x)$ lies in the $\tau^r$-orbit of $(0, x)$. Hence $(-d, x)=(rs, x)$ for some $s \in \Z$, and so $r$ divides $d$. On the other hand, since $\tau^r(0, x)=(-r, x)$ in $\Z\Delta$ and $(0, x)$ and $(-r, x)$ have the same $\tau^r$-orbits, we  have $\tau^r_{\CS} \widetilde{(0, x)}=\widetilde{(0, x)}$. This implies $d$ divides $r$. As both $r$ and $d$ are positive we get $r=d$. Now    $\Ga=\Z\mathbb{E}_n/<\tau^d>$, and it   implies that $\mid \Ga\mid=nd$.

{(2)} $\Delta=\mathbb{E}_6.$ If  $\alpha=\tau^r$ for some $r\geq 1$, then the same argument as the case $(1)$ works. Assume   $\alpha=\phi \tau^r$ for some $r \geq 1.$ Since ${\bf X}$ is a boundary vertex,  having exactly one immediate predecessor, equivalently exactly one immediate successor, $x$ is not possible one of  the vertices  2, 3, 5.  Assume $x$ is either $1$ or $6$.  Then the $\phi\tau^r$-orbit of $(0, x)$ consists of all pairs $(rs, 1), (rs, 6)$, where $s \in \Z$. Since $\tau_{\CS}^d\widetilde{(0, x)}=\widetilde{(0, x)}$, we observe that $(-d, x)=(rs, 1)$ or $(rs, 6)$ for some $s \in \Z$. So, for each of these cases we see that $r$ divides $d$. By the definition of $\phi$,  we see that the vertices $(j, 3)$ and $(j, 4)$, for any $j \in \Z$,  are fixed under the automorphism $\phi$. So the $\phi\tau^r$-orbits of $(0, 3)$ and $(0, 4)$ are the same as the $\tau^r$-orbits of $(0, 3)$ and $(0, 4)$, respectively.  By using this fact and applying  the same argument as in the case $(1)$ for the vertices $(0, 3)$ and $(0, 4)$, we obtain the equalities $\tau_{\CS}^r\widetilde{(0, 3)}=\widetilde{(0,3)}$ and $\tau_{\CS}^r\widetilde{(0, 4)}=\widetilde{(0,4)}$. Indeed, $r$ is the least number holding those conditions. As $\tau_{\CS} $ induces an automorphism on $\Ga$, hence those qualities force to have the equality $\tau^r_{\CS}{\bf X}={\bf X}$. It follows that $d$ divides $r$. Consequently, $r=d$, and so  $\Ga=\Z\mathbb{E}_6/<\phi\tau^d>$. This description of $\Ga$ yields  that $U(\widetilde{(0, 1)})=U(\widetilde{(0, 6)})$ and $U(\widetilde{(0, 2)})=U(\widetilde{(0, 5)})$ and
\begin{align*}
\mid\Ga\mid&=\mid U(\widetilde{(0, 1)})\mid+\mid U(\widetilde{(0,3)}) \mid +\mid U(\widetilde{(0, 4)})\mid +\mid U(\widetilde{(0, 2)})\mid\\
&= d+d+d+d \\
&= 4d.
\end{align*}

${(3)} \    \Delta=\mathbb{D}_n$ where $n\neq 4$. This case is proved in the sprit of the two cases we have already investigated. However, for the convenience of the reader, we repeat the main points of this case.

If  $\alpha=\tau^r$ for some $r\geq 1$, then again  the same argument as the case $(1)$ works. Assume that $\alpha=\tau^r\phi$ for some $r\geq 1$. If the vertex $x$ is neither $n$  nor $n-1$, then  the $\tau^r\phi$-orbit of $(0, x)$ is  the same as the $\tau^r$-orbit of $(0, x)$, as $\phi$ keeps the vertex $x$. Hence the same proof as in the case $(1)$ proves that $r=d$.  Assume that  $x$ is either $n$ or $n-1$. As we did in the case (2) for the vertices 1 and 6  therein, we have again the equality  $r$  and $d$. Finally, $\Ga=\Z\mathbb{D}_n/<\tau^d\phi>$. This implies that $U(\widetilde{(0, n)})=U(\widetilde{(0, n-1)})$ and
\[\mid\Ga\mid= \Sigma^{n-1}_{j=1} \mid U(\widetilde{(0, j)})\mid= (n-1)d. \]

{(4)} $\Delta=\mathbb{A}_n$ with $n$ even. Let $\alpha=\rho^r$ for some $r\geq 1$.

Keeping in mind that $\rho^2=\tau^{-1}$,  we get the vertices $(0, x)$ and $\rho^{2r}(0, x)=\tau^{-r}(0, x)$ are clearly in the same $\rho^r$-orbit. Thus $\tau^r_{\CS}\widetilde{(0, x)}=\widetilde{(0, x)}$. This implies that  $d$ divides $r$. As $\Ga=\Z\mathbb{A}_{n}/<\rho^r>$,  we observe that $U(\widetilde{(0, j)})=U(\widetilde{(0, n+1-j)})$ for every $1\leq j\leq n$, and
$$\mid\Ga \mid=\Sigma^{\frac{n}{2}}_{j=1}\mid U(\widetilde{(0, j)}) \mid=\frac{n}{2}r.$$
Hence 3 divides  $\mid \Ga \mid$, as required.

{(5)} $\Delta=\mathbb{A}_n$ with $n$ odd. We only consider the possibility $\alpha=\phi\tau^r$ for some $r \geq 1$, as the other possibility is proved in the same way as the case $(1)$.

Since $\phi^2$ is an automorphism of order 2, and moreover, the commutativity of $\tau$ and $S$, we infer that $(\phi\tau^r)^2(0, x)=\tau^{2r}(0, x).$ Hence $(0, x)$ and $\tau^{2r}(0, x)$ lie in the same $\phi\tau^2$-orbit. This leads that $\tau^{2r}_{\CS} {\bf X}=\tau^{2r}_{\CS}\widetilde{(0, x)}=\widetilde{(0, x)}={\bf X}$. Hence $d$ divides $2r$, and consequently $3$ divides $r$. As $\Ga=\Z\mathbb{A}_{n}/<\phi\tau^r>$,  we observe that $U(\widetilde{(0, j)})=U(\widetilde{(0, n+1-j)})$ for every $1\leq j\leq n$, and
$$\mid\Ga \mid=\Sigma^{\frac{n+1}{2}}_{j=1}\mid U(\widetilde{(0, j)}) \mid=\frac{n+1}{2}r.$$
This completes the proof.
\end{proof}

Note that the above proposition does not hold for type $\mathbb{D}_4$. For $\La=k[x]/(x^3)$, the stable Auslander-Reiten quiver $\Ga^s_{\CS}$ has only one component of type $\mathbb{D}_4$ with the cardinality 8, see e.g. \cite[Section 6]{RS2}.

\hspace{1 mm}



\section*{Acknowledgments}
The authors would like to thank the referees,  whose many useful comments significantly improved the paper.  This research is supported by the National Natural Science Foundation of China (Grant No. 12101316).


\begin{thebibliography}{9999}
\bibitem{Am} {\sc C. Amiot,} {\sl On the structure of triangulated categories with finitely many indecomposables,} {\em  Bull. Soc. Math. France} {\bf 135} (2007), 435-474.

\bibitem{AKM} {\sc S. Ariki, R. Kase and K. Miyamoto,} {\sl On components of stable Auslander-Reiten quivers that contain Heller lattices: The case of truncated polynomial rings}, {\em Nagoya Math. J.} {\bf 228} (2017), 72-113.

\bibitem{ASS} {\sc I. Assem, D. Simson and A. Skowro\'nski,}  {\sl Elements of the representation theory of assocative algebras I: Techniques of  representation theory}, London Mathematical Society Student Texts, {\bf 65}, Cambridge University Press, Cambridge, 2006.

\bibitem{AB} {\sc M. Auslander and M. Bridger,} {\sl  Stable module category,} {\em Mem. Amer. Math. Soc.} {\bf 94} (1969), 146pp.

\bibitem{AR} {\sc M. Auslander and I. Reiten,} {\sl  Stable equivalence of dualizing R-varietes,} {\em Adv. Math.} {\bf 12(3)} (1974),  306-366.

\bibitem{AR2} {\sc M. Auslander and I. Reiten,} {\sl Representation theory of Artin algebras. III. Almost split sequences,} {\em Comm. Algebra} {\bf 3} (1975), 239-294 .

\bibitem{AR3} {\sc M. Auslander and I. Reiten,} {\sl Representation theory of Artin algebras. IV. Invariants given by almost split sequences,} {\em Comm. Algebra} {\bf 5} (1977), 443-518.
	
\bibitem{ARS} {\sc M. Auslander, I. Reiten and S. O. Smal\o,} {\sl Representation theory of Artin algebras,} Cambridge Studies in Advanced Mathematics, {\bf 36}, Cambridge University Press, Cambridge, 1995.

\bibitem{AS} {\sc M. Auslander and  S. O. Smal\o,} {\sl  Almost split sequences in subcategories,} {\em J. Algebra} {\bf 69} (1981), 426-454.

\bibitem{Be} {\sc A. Beligiannis,} {\sl  On algebras of finite Cohen-Macaulay type,} {\em Adv. Math.} {\bf 226} (2011), 1973-2019.

\bibitem{BE} {\sc P. A. Bergh and K. Erdmann,} {\sl The stable Auslander-Reiten quiver of a quantum complete intersection,} {\em Bull. London Math. Soc.} {\bf 43} (2011), 79-90.

\bibitem{B} {\sc G. Birkhoff,} {\sl  Subgroups of abelian groups,} {\em Proc. Lond. Math. Soc.} II, {\bf 38} (1934), 385-401.

\bibitem{Bu} {\sc Th. B\"{u}hler,} {\sl Exact categories,} {\em Expo. Math.} {\bf 28} (1) (2010), 1-69.

\bibitem{CH} {\sc X. W. Chen,} {\sl The stable monomorphism category of a Frobenius category,} {\em Math. Res. Lett.} {\bf 18} (2011), 125-137.

\bibitem{CH2} {\sc X. W. Chen,} {\sl Gorenstein homological algebra of Artin algebras,} available at arXiv:1712.04587.

\bibitem{E} {\sc \"{O}. Eir\'{i}ksson,} {\sl From submodule categories to the stable Auslander algebra,} {\em J. Algebra} {\bf 486} (2017), 98-118.

\bibitem{EHS} {\sc H. Eshraghi, R. Hafezi and Sh. Salarian,} {\sl Total acyclicity for complexes of representations of quivers,} {\em Comm. Algebra} { \bf 41} (2013), 4425-4441.

\bibitem{EJ} {\sc E. E. Enochs and  O. M. G. Jenda,} {\sl  Gorenstein injective and projective modules,} {\em Math. Z.} {\bf 220} (1995), 611-633.

\bibitem{GKSP} {\sc N. Gao, J. K\"ulshammer, S. Kvamme and C. Psaroudakis,} {\sl A functorial approach to monomorphism categories for species. I,}  {\em  Commun. Contemp. Math.} {\bf 24} (2022), Paper No. 2150069, 55 pp.

\bibitem{GR} {\sc P. Gabriel and A.V. Roiter,} {\sl Representations of finite-dimensional algebras.} With a chapter by B. Keller. Encyclopaedia Math. Sci. {\bf 73}, Algebra, VIII, 1-177, Springer, Berlin, 1992.

\bibitem{H} {\sc R. Hafezi,} {\sl From subcategories to the entire module categories,}  {\em Forum Math.} {\bf 33} (2021), 245-270.

\bibitem{H2} {\sc R. Hafezi,} {\sl When stable Cohen-Macaulay Auslander algebra is semisimple,} available at arXiv:2109.00467.

\bibitem{HM} {\sc R. Hafezi and I. Muchtadi,} {\sl  Different  exact structures  on the monomorphism categories,}  {\em Appl. Categ. Str.} {\bf 29} (2021), 31-68.

\bibitem{Ha} {\sc D.  Happel,} {\sl  Triangulated categories in the representation theory of finite Dimensional Algebras,}  London Math. Soc., Lecture Notes Ser. {\bf 119}, Cambridge University Press, Cambridge, 1988.

\bibitem{HLXZ}{\sc W. Hu, X. H. Luo, B. L. Xiong and G. D. Zhou,} {\sl  Gorenstein projective bimodules via monomorphism categories and filtration categories,}  {\em J. Pure Appl. Algebra} {\bf 223} (2019), 1014-1039.

\bibitem{INP} {\sc O. Iyama,  H. Nakaoka and  Y. Palu,} {\sl Auslander-Reiten theory in extriangulated categories,} available at arXiv:1805.03776.

\bibitem{K} {\sc B. Keller,} {\sl Chain complexes and stable categories,} {\em Manuscripta Math.} {\bf 67} (1990), 379-417.

\bibitem{KS} {\sc H. Krause and \o. Solberg,} {\sl  Applications of cotorsion pairs,} {\em J. London Math. Soc.} {\bf (2) 68} (2003), 631-650.

\bibitem{LZ}{\sc Z. W. Li and P. Zhang,} {\sl   A construction of Gorenstein-projective modules,} {\em J. Algebra} {\bf 323}, (2010), 1802-1812.

\bibitem{Li} {\sc S. Liu,} {\sl Auslander-Reiten theory in a Krull-Schmidt category,}  {\em Sao Paulo J. Math. Sci.} {\bf 4} (2010), 425-472.

\bibitem{LNP} {\sc S. Liu, P. Ng and C. Paquette,} {\sl Almost split sequences and approximations,} {\em Algebr. Represent. Theory} {\bf 16} (2013), 1809-1827.

\bibitem{Lu} {\sc M. Lu,} {\sl Gorenstein properties of simple gluing algebras,}  {\em Algebr. Represent. Theory} {\bf 22} (2019),  517-543.

\bibitem{LuZ} {\sc X. H. Luo and P. Zhang,} {\sl  Monic representations and Gorenstein-projective modules,}  {\em Pacific J. Math.} {\bf 264} (2013), 163-194.
   	
\bibitem{Q} {\sc D. Quillen,} {\sl Higher algebraic $K$-theory. I,} in: Algebraic $K$-theory, I: Higher $K$-theories, Proceedings of the Conference, Battelle Memorial Institute, Seattle, Washington, 1972, Lecture Notes in Mathematics, {\bf 341}, Springer, Berlin, (1973), 85-147.

\bibitem{Rie} {\sc C. Riedtmann,} { \sl Algebren, {D}arstellungsk\"ocher, \"{U}berlagerungen und zur\"uck}, {\em Comment. Math. Helv.} \textbf{55} (1980), 199--224.

\bibitem {Riee}  {\sc C. Riedtmann,} {\sl Many algebras with the same Auslander-Reiten quiver,} {\em Bull. London Math. Soc.} {\bf 15} (1983), 43-47.

\bibitem{Rie1} {\sc C. Riedtmann,} {\sl Representation-finite selfinjective algebras of class $D_n$,} {\em Compositio Math.} {\bf 49} (1983), 231-282.

\bibitem{Ri} {\sc C. M. Ringel,} {\sl  Finite-dimensional hereditary algebras of wild representation type,}  {\em Math. Z.} {\bf 161} (1978), 235-255.

\bibitem{RS3} {\sc C. M. Ringel and M. Schmidmeier,} {\sl  Submodule categories of wild representation type,}  {\em J. Pure Appl. Algebra} {\bf 205}, (2006), 412-422.

\bibitem{RS} {\sc C. M. Ringel and M. Schmidmeier,} {\sl  The Auslander-Reiten translation in submodule categories,} {\em Trans. Amer. Math. Soc.} {\bf 360} (2008), 691-716.

\bibitem{RS2} {\sc C. M. Ringel and M. Schmidmeier,} {\sl   Invariant subspaces of nilpotent linear operators,} I. {\em J. Reine Angew. Math.} {\bf 614} (2008), 1-52.

\bibitem{RZ} {\sc C. M. Ringel and P. Zhang,} {\sl   From submodule categories to preprojective algebras,}  {\em Math. Z.} {\bf278} (2014), 55-73.

\bibitem{S} {\sc D. Simson,} {\sl Representation types of the category of subprojective representations of a finite poset over $K[t]/(t^m)$ and a solution of a Birkhoff type problem,} {\em J. Algebra} {\bf 311} (2007), 1-30.

\bibitem{SY} {\sc A. Skowro\'nski and K. Yamagata,} {\sl  Frobenius algebras. I. Basic representation theory,} EMS Textbooks in Mathematics. European Mathematical Society (EMS), Z\"urich, 2011.

\bibitem{T} {\sc JN. Thiel,} {\sl On Euclidean Components of Auslander-Reiten Quivers of Finite Group Schemes,} {\em Algebr. Represent. Theory} (2022). https://doi.org/10.1007/s10468-022-10130-9.

\bibitem{WL} {\sc Z. P. Wang and Z. K. Liu,} {\sl Recollements induced by monomorphism categories,} {\em J. Algebra} {\bf 594} (2022), 614-635.

\bibitem{XZ} {\sc J. Xiao and B. Zhu,} {\sl Locally finite triangulated categories,} {\em J. Algebra} {\bf 290} (2005), 473-490.

\bibitem{XZZ}{\sc B. L. Xiong, P. Zhang and  Y. H. Zhang,} {\sl  Auslander-Reiten translations in monomorphism categories,} {\em Forum Math.} {\bf 26} (2014),  863-912.

\bibitem{X} {\sc J. Xu,} {\sl Flat covers of modules,} Lecture Notes in Mathematics, {\bf 1634}, Springer-Verlag, Berlin, 1996.

\bibitem{Z} {\sc P. Zhang,} {\sl  Monomorphism categories, cotilting theory, and Gorenstein-projective modules,} {\em J. Algebra} {\bf 339} (2011), 181-202.

\bibitem{ZX} {\sc  P. Zhang and B.-L. Xiong,} {\sl Separated monic representations II: Frobenius subcategories and RSS equivalences,} {\em Trans. Amer. Math. Soc.} {\bf 372} (2019), 981-1021.
\end{thebibliography}
\end{document}